\documentclass[12pt, reqno]{amsart}
\usepackage{amssymb,amsthm,amsmath}
\usepackage[numbers,sort&compress]{natbib}
\usepackage{amssymb,amsmath}
\usepackage{amsfonts,dsfont}
\usepackage{mathrsfs}
\usepackage{latexsym}
\usepackage{amssymb}
\usepackage{amsthm}
\usepackage{indentfirst}
\usepackage{hyperref}
\usepackage{color}
\date{Aug 13, 2024
}

\let\oldsection\section
\renewcommand\section{\setcounter{equation}{0}\oldsection}

\newtheorem{corollary}{Corollary}[section]
\newtheorem{theorem}{Theorem}[section]
\newtheorem{lemma}{Lemma}[section]
\newtheorem{proposition}{Proposition}[section]
\newtheorem{definition}{Definition}[section]

\newtheorem{remark}{Remark}[section]

\allowdisplaybreaks
\begin{document}

\title[ existence and uniqueness of primitive equations]{Global well-posedness of the 3D primitive equations with horizontal viscosity and
vertical diffusivity II: close to $H^1$ initial data}

\author{Chongsheng~Cao}
\address[Chongsheng~Cao]{Department of Mathematics, Florida International University, University Park, Miami, FL 33199, USA}
\email{caoc@fiu.edu}

\author{Jinkai~Li}
\address[Jinkai~Li]{South China Research Center for Applied Mathematics and Interdisciplinary Studies, School of Mathematical Sciences,
South China Normal University, Guangzhou 510631, China}
\email{jklimath@gmail.com; jklimath@m.scnu.edu.cn}

\author{Edriss S.~Titi}
\address[Edriss S.~Titi]{Department of Mathematics, Texas A$\&$M University, College Station, TX 77843, USA.
Department of Applied Mathematics and Theoretical Physics, University of Cambridge,
Cambridge CB3 0WA, UK. Department of Computer Science and Applied Mathematics,
Weizmann Institute of Science, Rehovot 76100, Israel.}
\email{titi@math.tamu.edu; Edriss.Titi@maths.cam.ac.uk}

\author{Dong~Wang}
\address[Dong~Wang]{South China Research Center for Applied Mathematics and Interdisciplinary Studies, School of Mathematical Sciences,
South China Normal University, Guangzhou 510631, China}
\email{WDongter@m.scnu.edu.cn}

\keywords{Global well-posedness; primitive equations; Vertical eddy diffusivity; Horizontal eddy viscosity.}
\subjclass[2010]{26D10, 35Q35, 35Q86, 76D03, 76D05, 86A05, 86A10.}


\begin{abstract}
In this paper, we consider the initial-boundary value problem to the
three-dimensional primitive equations for the oceanic and atmospheric dynamics with
only horizontal eddy viscosities in the horizontal momentum equations and only vertical
diffusivity in the temperature equation in the domain $\Omega=M\times(-h,h)$, with $M=(0,1)\times(0,1)$. Global well-posedness of strong solutions is established, for any initial data $(v_0,T_0) \in H^1(\Omega)\cap L^\infty(\Omega)$ with
$(\partial_z v_0, \nabla_H T_0) \in L^q(\Omega)$ and $v_0 \in L_z^1(B^1_{q,2}(M))$, for some
$q \in (2,\infty)$, by using delicate energy estimates and maximal regularity estimate in the anisotropic setting.
\end{abstract}

\maketitle


\section{Introduction}
\label{sec1}
\allowdisplaybreaks
The incompressible primitive equations are derived from the Boussinesq equations of incompressible
flows and form a fundamental block in models of planetary oceanic and atmospheric dynamics;
see, e.g., Lewandowski \cite{Lewan}, Majda \cite{Majda}, Pedlosky \cite{Pedlosky}, Vallis \cite{Vallis}, Washington-Parkinson \cite{Washing}, and Zeng \cite{Zeng}.
In the context of the large-scale oceanic and atmospheric dynamics, an important feature is that the vertical scale of the ocean
and atmosphere is much smaller than the horizontal ones. Due to this fact, the primitive equations are derived from the Navier-Stokes equations by taking the small aspect
ratio limit. Such limit has been rigorous justified, for the case of incompressible fluid, first
by Az\'{e}rad-Guill\'{e}n \cite{PA.FG} in the framework of weak solutions without error estimate; and with error estimate for strong solutions
by Li-Titi \cite{J.Li2,J.Li3}. See also follow up related results by Furukawa-Giga-Hieber-Hussein-Kashiwabara-Wrona \cite{FugigaHHK, FuGiHiHuKaW1}, Furukawa-Giga-Kashiwabara \cite{FuruGigaKas}, Li-Titi-Yuan \cite{LiTitiYuan}, and Pu-Zhou \cite{Puzhou}. On the other hand the justification of such a limit and error estimates, in the atmospheric compressible case, are more involved and delicate depending on the structure of the initial data (see
\cite{LiuTiti} and \cite{GaoNeTang}). Moreover, in the oceanic and atmospheric
dynamics, due to the strong horizontal turbulent mixing, the horizontal eddy viscosity is much stronger than the vertical one; the vertical
eddy viscosity is very weak and thus often neglected.

In this paper, we consider the following system of primitive equations, which have only horizontal viscosity and vertical diffusivity:
\begin{eqnarray}
  &\partial_tv+(v\cdot\nabla_H)v+w\partial_zv+\nabla_Hp-\Delta_Hv+f_0\vec{k}\times v=0,\label{eq1}\\
  &\partial_zp+T=0,\label{eq2}\\
  &\nabla_H\cdot v+\partial_zw=0,\label{eq3}\\
  &\partial_tT+v\cdot\nabla_HT+w\partial_zT-\partial_z^2T=0,\label{eq4}
\end{eqnarray}
where the horizontal velocity $v=(v^1,v^2)$, vertical velocity $w$, pressure $p$, and temperature $T$ are the unknowns,
and $f_0$ is the Coriolis parameter. Here $\nabla_H=(\partial_x, \partial_y)$ and $\Delta_H=\partial_x^2+\partial_y^2$
denote the horizontal gradient and the horizontal Laplacian, respectively, and $\vec{k}\times v=(-v^2, v^1)$.

Mathematical studies on the fully viscous three-dimensional primitive equations were started by Lions-Temam-Wang \cite{Lions1,Lions2,Lions3} in the 1990s, and they established the global existence of weak solutions; however, the uniqueness of weak solutions is still an open problem, even for the two-dimensional
case. Note that, unlike the primitive equations, it is well known that weak solutions to the two-dimensional incompressible
Navier-Stokes equations are unique, see, e.g., Constantin-Foias \cite{PccF}, Ladyzhenskaya \cite{lady} and Temam \cite{Teman}. Although the
general uniqueness of weak solutions to the primitive equations is still unknown, some special cases have been addressed, see \cite{J.Li4,DFA,IYWM,TTM,LiYuan}. Remarkably, also unlike the
three-dimensional Navier-Stokes equations for which the global well-posedness is still unknown in general, global well-posedness of strong solution to the three-dimensional primitive equations has clearly been known since the work by Cao-Titi \cite{Cao1}. This global existence of strong solutions to the primitive equations was also proven later by Kobelkov \cite{GMK} and Kukavica-Ziane \cite{IKMZ}, by using different arguments and for different boundary conditions, see also Guo-Huang \cite{Boling}, Hieber-Kashiwabara \cite{MHTKa} and Hieber-Hussien-Kashiwabara \cite{MHATK} for some generalizations in the $L^p$ setting.

Recalling that the horizontal viscosity and diffusivity are much stronger than the vertical ones, it is natural to consider the primitive equations with anisotropic viscosity or diffusivity. While in all the papers mentioned above, the systems considered  there are assumed to have both full viscosity and full diffusivity. Mathematical studies on the primitive equations with partial viscosity or partial diffusivity were carried out
by Cao-Titi \cite{Cao2} and Cao-Li-Titi \cite{J.Li5,J.Li6}, where they proved that the primitive equations with full viscosity and with in addition either horizontal diffusivity or vertical diffusivity in the temperature equation have a unique global strong solution. Surprisingly, the vertical viscosity is not necessary for the global well-posedness of the primitive equations. In fact, it was proved by Cao-Li-Titi \cite{J.Li1,J.Li7, J.Li8} that the primitive equations with only horizontal viscosity are globally well-posed, as long as one still has either horizontal or vertical diffusivity.

In this paper, we continue the study of the primitive equations with only horizontal viscosity and only vertical diffusivity, i.e., system (\ref{eq1})--(\ref{eq4}), which has already been studied in \cite{J.Li1} for regular enough initial data. The aim of this paper is to improve the results in \cite{J.Li1} by relaxing the regularities of the initial data.

Consider system (\ref{eq1})--(\ref{eq4}) on the domain $\Omega := M\times (-h,h)$, with $M = (0,1)\times(0,1)$, and supplement it  with the following boundary and initial conditions:
\begin{eqnarray}
  &v, w, p, T \text{are periodic in}\, x, y, z, \label{eq5}\\
  &v\, \text{and}\, p\, \text{are even in}\, z, w\, \text{and}\, T \,\text{are odd in}\, z,  \label{eq6}\\
  &(v,T)|_{t=0}=(v_0,T_0). \label{eq7}
\end{eqnarray}
It should be noticed that, as pointed out in \cite{J.Li1},
the periodic and symmetry boundary conditions
(\ref{eq5})--(\ref{eq6}) to system (\ref{eq1})--(\ref{eq4}) on the domain $M\times(-h, h)$ are equivalent to
the physical boundary conditions of no-permeability and stress-free at the solid physical boundaries $z=-h$ and $z=0$ in the sub-domain
$M\times(-h,0)$, namely:
\begin{eqnarray*}
&  v, w, p, T \mbox{ are periodic in }x\mbox{ and }y, \\
&(\partial_zv,w)|_{z=-h,0}=0,\quad T|_{z=-h,0}=0.
\end{eqnarray*}

Then, system (\ref{eq1})--(\ref{eq4}), subject to (\ref{eq5})--(\ref{eq7}) is equivalent to (see \cite{J.Li5,J.Li6})
\begin{align}
\partial_tv+&(v\cdot\nabla_H)v+w\partial_zv-\Delta_Hv+f_0\vec{k}\times v \nonumber\\
+&\nabla_H\left(p_s(x,y,t)-\int_{-h}^zT(x,y,\xi,t)d\xi\right)=0, \label{eq8}\\
&\int_{-h}^h\nabla_H\cdot v(x,y,z,t)dz=0, \label{eq9}\\
\partial_tT+&v\cdot\nabla_HT+w\partial_zT-\partial_z^2T=0, \label{eq10}
\end{align}
where $p_s(x,y,t)$ is the pressure at $z=-h$ and
\begin{align}
&w(x,y,z,t)=-\int_{-h}^z \nabla_H \cdot v(x,y,\xi,t)d\xi, \label{eq11}
\end{align}
subject to the following boundary and initial conditions
\begin{eqnarray}
  &v, T \,\text{are periodic in}\, x, y, z, \label{eq12}\\
  &v \,\text{and}\, T \,\text{are even and odd in}\, z, \text{respectively},  \label{eq13}\\
  &(v,T)|_{t=0}=(v_0,T_0). \label{eq14}
\end{eqnarray}
Noticing that the symmetry condition (\ref{eq13}) is propagated by system (\ref{eq8})--(\ref{eq10}), it suffices to assume it on
the initial data.

Applying the operator $\text{div}_H$ to (\ref{eq8}) and integrating with respect to $z$ over $(-h,h)$, one can see that $p_s(x,y,t)$
satisfies the following
\begin{align}
\label{eq15}
\begin{cases}
&-\Delta_H p_s=\frac{1}{2h}\nabla_H\cdot\int_{-h}^h(\nabla_H\cdot(v\otimes v)+f_0\vec{k}\times v-\int_{-h}^z \nabla_HTd\xi)dz,\\
&\int_M p_s(x,y,t)dxdy=0, \,\,\,\,p_s \,\text{is periodic in}\, x \,\text{and}\, y.
\end{cases}
\end{align}
Here the constraint $\int_M p_sdxdy=0$ is imposed to guarantee the uniqueness of the solution $p_s$ to the elliptic problem (\ref{eq15}).

Throughout this paper, $L^q(M)$, $L^q(\Omega)$, $W^{m,q}(M)$, and $W^{m,q}(\Omega)$ denote the standard Lebesgue and Sobolev spaces, repectively, and $B_{p,q}^s(M)$ is the Besov space. For $q=2$, we use $H^m$ to replace $W^{m,2}$. We always use $\|f\|_p$ to denote the $L^p(\Omega)$ norm of $f$, while use
$\|f\|_{p,M}$ to denote the $L^p(M)$ norm of $f$.

\begin{definition}\label{def1}
Given a positive time $\mathcal{T}$ and periodic function $ (v_0, T_0) \in H^1(\Omega) $. Assume that $\int_{-h}^{h} \nabla_H\cdot v_0 dz=0,$ where
$v_0 \in L_z^1(B^1_{q,2}(M))$, $\partial_z v_0,\nabla_H T \in L^q(\Omega)$, for some $q \in (2,\infty)$, and that $v_0$ and $T_0$
are even and odd in $z$, respectively.
A pair $(v,T)$ is called a strong solution to system
(\ref{eq8})--(\ref{eq10}), subject to (\ref{eq12})--(\ref{eq14}), on
$\Omega\times(0,\mathcal T)$, if

(i) $v$ and $T$ are periodic in $x,y,z$, and are even and odd in $z$, respectively;

(ii) $v$ and $T$ have the regularities
\begin{eqnarray*}
  &&v \in L^{\infty}(0,\mathcal{T} ;H^1(\Omega))\cap C([0,\mathcal{T}];L^2(\Omega)),\\
  &&T \in L^{\infty}(0,\mathcal{T} ;H^1(\Omega)\cap L^{\infty}(\Omega))\cap C([0,\mathcal{T}];L^2(\Omega)),\\
  &&\partial_z v, \nabla_H T \in L^{\infty}(0,\mathcal{T};L^q(\Omega)),\\
  &&\nabla_H v, \partial_z T \in L^2(0,\mathcal{T};H^1(\Omega)),\quad
\partial_t v, \partial_t T\in L^2(0,\mathcal{T};L^2(\Omega));
\end{eqnarray*}

(iii) $v$ and $T$ satisfy (\ref{eq8})--(\ref{eq10}) a.e.\,in $\Omega \times (0,\mathcal{T})$, with $w$ and $p_s$ defined by (\ref{eq11}) and
(\ref{eq15}), respectively, and satisfy the initial condition (\ref{eq14}).
\end{definition}

\begin{definition}\label{def2}
A pair $(v,T)$ is called a global strong solution to system (\ref{eq8})--(\ref{eq10}), subject to (\ref{eq12})--(\ref{eq14}), if it is a strong solution
on $\Omega\times(0,\mathcal{T})$, for any $\mathcal{T} \in (0,\infty)$.
\end{definition}

\begin{theorem}
  \label{the1}

Let periodic function $(v_0, T_0) \in H^1(\Omega)\cap L^{\infty}(\Omega)$, with $\int_{-h}^{h} \nabla_H\cdot v_0dz=0$, $
v_0 \in L_z^1(B^1_{q,2}(M))$, and $\big(\partial_z v_0,\nabla_H T_0\big) \in L^q(\Omega)$, for some $q \in (2,\infty)$. Assume that $v_0$ and $T_0$
are even and odd in $z$, respectively.

Then, system (\ref{eq8})--(\ref{eq10}), subject to (\ref{eq12})--(\ref{eq14}), has a unique global strong solution $(v,T)$, which is continuously depending on the initial data.
\end{theorem}

\begin{remark}
Comparing with the previous work \cite{J.Li1}, where global well-posedness of strong solutions to the same problem as in this paper 
was established under the regularity conditions on the initial data that $v_0\in H^2(\Omega)$, $T_0\in H^1(\Omega)\cap L^\infty(\Omega)$, and $\nabla_HT_0\in L^q(\Omega)$, for some $q\in(2,\infty)$, less regularities on the initial data are required in Theorem \ref{the1} and, in particular, we do not need any information about the two derivatives of $v_0$ any more. Therefore, the result in Theorem \ref{the1} extends that in \cite{J.Li1}
by relaxing the regularities on the initial data.
\end{remark}

There are two main difficulties in proving the theorem. First is the ``mismatching" of regularities between the horizontal momentum equations
(\ref{eq8}) and the temperature equation (\ref{eq10}); and second is the absence of horizontal dissipation term in the temperature
equation. To overcome the ``mismatching" of regularities, in the spirit of \cite{J.Li1}, we employ the auxiliary functions $\theta$ and $\eta$,
being given in (\ref{eq16}) below,
to obtain the a priori $L_t^{\infty}(L^2)$ type estimate on the horizontal derivatives of the horizontal velocity. In
principle, the auxiliary functions $\theta$ and $\eta$ are used to control the horizontal derivatives of the velocity. The hardest a priori
estimate to be derived in this paper is the $L^\infty_t(L^q)$ estimate for $\nabla_HT$. Due to the absence of the horizontal dissipation in the
temperature equation, the following two terms will be encountered
$$
  \int_{\Omega} |\nabla_Hw||\partial_z\nabla_H T||\nabla_H T|^{q-2} dxdydz\quad\mbox{and}\quad
  \int_{\Omega}|\nabla_Hv| |\nabla_H T|^{q}dxdydz.
  $$
To deal with the first term, above, we make use of the $L_z^1L^2_tL_M^q$ type anisotropic parabolic estimate for $\nabla_H^2v$,
which is obtained by applying the maximal
regularity theory to the momentum equations in $x,y,t$ coordinates first and then integrating in the
$z$ variable. Concerning the second term, above, for
the sake of short in time estimates, we again
reduce it to the $L_z^1L^2_tL_M^q$ type anisotropic parabolic estimate for $\nabla_H^2v$; while for the sake
of long time estimate (but away from initial time),
we decompose the velocity into the ``temperature-determined" part $v_T$ and the ``velocity-dominated" part $v_s$, and
control the corresponding terms related to $v_T$ and $v_s$ in different ways, where $v_s$ and $v_T$, respectively, are being defined by (\ref{eq40}) and (\ref{eq41}), below.
Concerning the term related to $v_T$, thanks to the uniform boundedness of the temperature,
we apply the Beale-Kato-Majda type logarithmic Sobolev
inequality in \cite{BKM}
to deal with it. For the term related to $v_s$, we need to get the a priori $L^2_tL^2_zL^\infty_M$ estimate for $\nabla_Hv_s$, which
is achieved by carrying out the $L^2_tL^2_zL^q_M$ type anisotropic parabolic
estimate on $\nabla_H(\eta,\theta)$. Comparing to the previous work \cite{J.Li1}, the main
contribution of this paper is that the a priori $L^\infty_tL^q$ estimate of $\nabla_HT$ is achieved without appealing to the a priori
$L^\infty_tL^2$ estimate on $\partial_z^2v, \nabla_H\partial_zv, \nabla_H\eta, \nabla_H\theta$ and, as a result, we do not need any information
about the two derivatives of the initial velocity, which in turn allows us to assume less regular initial data.

The rest of this paper is arranged as follows: in the next section, Section \ref{sec2}, we collect some preliminary results which will be used in the later sections;
in Section \ref{sec3}, which is the key part of this paper, we establish the a priori estimate to strong solutions, containing successively
the $L^\infty(H^1)$ estimate on $v$, anisotropic parabolic estimates on $v, \eta, \theta$, short in time and global in time $L^\infty(L^q)$ estimates on $\nabla_HT$, and finally the overall estimates; in Section \ref{sec4}, we give the proof of Theorem \ref{the1}.

Throughout this paper, the letter $C$ denotes a general positive constant which may vary from line to line.

\section{Preliminaries}\label{sec2}

In this section, we state several preliminary lemmas which will be used
in the rest of this paper.

\begin{lemma}[see Lemma 2.1 in \cite{J.Li1}]\label{lemma1}

The following inequality holds:
\begin{eqnarray*}
  &&\int_M\left(\int_{-h}^h|f|dz\right)\left(\int_{-h}^h |gw|dz\right)dxdy \\
  &\leq& C\|f\|_2\|g\|_2^{\frac{1}{2}}\left(\|g\|_2^{\frac{1}{2}}+\|\nabla_Hg\|_2^{\frac{1}{2}}\right)\|w\|_2^{\frac{1}{2}}\left(\|w\|_2^{\frac{1}{2}}+\|\nabla_Hw\|_2^{\frac{1}{2}}\right)
\end{eqnarray*}
for every $f$, $g$ and $w$ such that the right-hand sides make sense and are finite, where $C$ is a positive constant depends only on $h$.
\end{lemma}

\begin{lemma}[see Lemma 2.2 in \cite{J.Li1}]\label{lemma2}

The following inequality holds:
$$
\int_M\left(\int_{-h}^h|f|dz\right)\left(\int_{-h}^h |gw|dz\right)dxdy
  \leq\left(\int_{-h}^h \|f\|_{4,M}dz\right)\left(\int_{-h}^h \|g\|_{4,M}^2dz\right)^{\frac{1}{2}}\|w\|_2
$$
and
$$
\int_M\left(\int_{-h}^h|f|dz\right)\left(\int_{-h}^h |gw|dz\right)dxdy\leq \|f\|_2\left(\int_{-h}^h \|g\|_{4,M}^2dz\right)^{\frac{1}{2}}\left(\int_{-h}^h \|w\|_{4,M}^2dz\right)^{\frac{1}{2}}
$$
for every $f$, $g$ and $w$ such that the right-hand sides make sense and are finite.
\end{lemma}

The following lemma can be proven in the similar way as in Beale-Kato-Majda \cite{BKM}.

\begin{lemma}
 \label{lemma4}
For any periodic two-dimensional vector field $g=(g_1,g_2) \in W^{1,q}(M),$ with $q\in (2,\infty)$, it holds that
$$
  \|\nabla_H g\|_{L^\infty(M)}\leq C(\|\nabla_H\cdot g\|_{L^\infty(M)}+\|\nabla_H^{\bot}\cdot g\|_{L^\infty(M)}+1)
  \log(e+\|\nabla_H g\|_{W^{1,q}(M)})
$$
for a positive constant $C$ depending only on $q$.
\end{lemma}

\begin{lemma}
  \label{pro3.2}
The following estimate holds
\begin{align*}
\sup_{-h\leq z\leq h} \|v(.,z,t)\|_{4,M}^4 \leq C(\|v\|_4^4 +\|\partial_z v\|_2\|v\|_6^3),
\end{align*}
where $C$ is a positive constant depending only on $h$.
\end{lemma}

\begin{proof}

 For any $z \in (-h, h)$, it follows from the H\"{o}lder inequality that
\begin{align*}
\|v(.,z,t)\|_{4,M}^4 &\leq \frac{1}{2h}\int_{-h}^h \|v(.,z,t)\|_{4,M}^4dz + 4\int_{-h}^h\int_M |v|^3|\partial_z v|dxdydz\\
&\leq C(\|v\|_4^4 +\|\partial_z v\|_2\|v\|_6^3),
\end{align*}
leading to the conclusion.
\end{proof}

The following logarithmic type Gr\"{o}nwall inequality in the same spirit of \cite{Li-Titi-L,Li-TitiA,J.Li1} will be used later.

\begin{lemma} \label{lemma5}
Given $\mathcal{T} \in (0,\infty)$. Let $A$, $B$, $m$, $f$, $\ell$ be measurable nonnegative functions defined on $(0,\mathcal{T})$, with $A\geq1$
being absolutely continuous on $(0,\mathcal{T})$ and $f,\ell\in L^1\big((0, \mathcal{T})\big)$, satisfying
\begin{equation*}
  \left\{
  \begin{array}{l}
     \displaystyle A'(t)+B(t) \leq m(t)+A\left[\log A+\log(A+B+e)+f(t)\right]+\ell(t),\quad\forall t\in(0,\mathcal T),  \\
   \displaystyle \int_0^t m(s)ds \leq \int_0^t A(s)ds, \quad \forall t \in (0,\mathcal{T}).
  \end{array}
  \right.
\end{equation*}
Then, it holds that
\begin{eqnarray*}
  A(t)+\int_0^t B(s)ds\leq 2e^{e^{t}\left[\log \left(A(0)+e\right)+\int_0^t(f(s)+l(s)+log2+2)ds\right]},\quad\forall t\in(0,\mathcal T).
\end{eqnarray*}
\end{lemma}

\begin{proof}

Set $\tilde{A}=A+e$ and $\tilde{B}=A+B+e$. Since $\log z\leq log(z+1)\leq z$ for any $z>0$, one deduces by assumption that
\begin{eqnarray*}
\frac{\mathrm{d}}{\mathrm{d}t}\tilde{A}(t)+\tilde{B}(t) &\leq& m(t)+\tilde{A}\left(\log \tilde{A}+\log\tilde{B}+f(t)+1\right)+\ell(t)\\
&=&m(t)+\tilde{A}\left[\log \tilde{A}+\log\left(\frac{\tilde{B}}{2\tilde{A}}\right)+\log(2\tilde{A})+f(t)+1\right]+\ell(t)\\
&\leq& \frac{\tilde{B}}{2}+m(t)+\tilde{A}\left(2\log \tilde{A}+f(t)+1+\log2\right)+\ell(t),
\end{eqnarray*}
which implies
\begin{align*}
&\frac{\mathrm{d}}{\mathrm{d}t}\tilde{A}(t)+\frac{\tilde{B}(t)}{2}
\leq m(t)+\tilde{A}\left(2\log \tilde{A}+f(t)+1+\log2\right)+\ell(t).
\end{align*}
Thanks to this and since $\int_0^t m(s)ds \leq \int_0^t A(s)ds$, one obtains
\begin{eqnarray*}
 \tilde{A}(t)+\int_0^t\frac{\tilde{B}(s)}{2}ds
 &\leq& \int_0^t \Big(2\log \tilde{A}+f(s)+2+\log2\Big)\tilde{A}(s)ds+\int_0^t \ell(s)ds +A(0)\\
 &=:&F(t)
\end{eqnarray*}
and thus
\begin{eqnarray*}
 F'(t)
 &=&\Big(2\log \tilde{A}+f(t)+2+\log2\Big)\tilde{A}(t)+\ell(t)\\
 &\leq& 2F(t)\log F(t)+f(t)F(t)+(2+\log2)F(t)+\ell(t).
\end{eqnarray*}
Dividing both sides of the above inequality by $F(t)$ yields
$$\frac{d}{dt}\log F(t)\leq \log F(t)+f(t)+\ell(t)+2+\log2,$$
where $F(t)\geq1$ was used. By the Gr\"{o}nwall inequality, one obtains
$$\log F(t)\leq e^{t}\left(\log F(0)+\int_0^t \left(f(s)+\ell(s)+2+\log2\right)ds\right)$$
and thus
\begin{eqnarray*}
F(t)&\leq& e^{e^{t}\left(\log F(0)+\int_0^t \left(f(s)+\ell(s)+2+\log2\right)ds\right)}\\
&=&e^{e^{t}\left(\log( A(0)+e)+\int_0^t \left(f(s)+\ell(s)+2+\log2\right)ds\right)}.
\end{eqnarray*}
With the aid of this and recalling the definition of $F(t)$, it follows that
\begin{eqnarray*}
A(t)+\int_0^t B(s)ds&\leq& 2\left(\tilde{A}(t)+\int_0^t \frac{\tilde{B}(s)}{2}ds\right)\\
&\leq& 2e^{e^{t}\left(\log( A(0)+e)+\int_0^t \left(f(s)+\ell(s)+2+\log2\right)ds\right)},
\end{eqnarray*}
hence the conclusion follows.
\end{proof}

The following logarithmic type anisotropic Sobolev inequality is cited from \cite{J.Li8}, some relevant inequalities can be found in \cite{Caowu,CaoFarTiti,Danpai}.
\begin{lemma}[see Lemma 2.4 in \cite{J.Li8}]\label{lemma29} Let $\lambda>0$ and $\mathbf{p}=(p_1, p_2, p_3)$, with $p_i \in (1, \infty)$ and $\frac{1}{p_1}+\frac{1}{p_2}+\frac{1}{p_3}<1$. Then, for any periodic function $F$ on $\Omega$, we have
\begin{eqnarray*}
\|F\|_\infty\leq C_{\mathbf{p}, \lambda, \Omega}\max\left\{1, \sup_{r\geq 2}\frac{\|F\|_r}{r^{\lambda}}\right\}\log^{\lambda}\left(\sum_{i=1}^3\left(\|F\|_{p_i}+\|\partial_i F\|_{p_i}\right)+e\right).
\end{eqnarray*}
\end{lemma}

 We also need the following Aubin-Lions Lemma.
\begin{lemma}[Aubin-Lions Lemma, See Corollary 4 in \cite{JSimon}]\label{lemma7}
Assume that $X$, $B$ and $Y$ are three Banach spaces, with $X\hookrightarrow\hookrightarrow B\hookrightarrow Y$. Then the following two items hold:

(i) If $\mathcal{F}$ is a bounded subset of $L^p(0,\mathcal{T};X)$, where $1\leq p< \infty$,
and $\frac{\partial \mathcal{F}}{\partial t} :=\left\{\frac{\partial f}{\partial t} \big| f \in \mathcal{F}\right\}$ is bounded in
$L^1(0,\mathcal{T};Y)$, then $\mathcal{F}$ is relatively compact in $L^p(0,\mathcal{T};B)$;

(ii) if $\mathcal{F}$ is a bounded subset of $L^{\infty}(0,\mathcal{T};X)$, and $\frac{\partial \mathcal{F}}{\partial t}$
is bounded in $L^r(0,\mathcal{T};Y)$, where $r>1$, then $\mathcal{F}$ is relatively compact in $C([0,\mathcal{T}];B)$.
\end{lemma}

\section{A priori estimate }\label{sec3}

This section is devoted to establishing the a priori estimates on suitably solutions to system (\ref{eq8})--(\ref{eq10}), subject to (\ref{eq12})--(\ref{eq14}). We start with the following global well-posedness result proved in Cao-Li-Titi \cite{J.Li1}.

\begin{proposition}[see Theorem 1.1 in Cao-Li-Titi \cite{J.Li1}]
  \label{prop2.1}
Suppose $v_0 \in H^2(\Omega)$, $T_0 \in H^1(\Omega)\cap L^{\infty}(\Omega)$, with $\int_{-h}^h \nabla_H \cdot v_0(x,y,z)dz=0$ and $\nabla_H T_0 \in L^q(\Omega)$, for some $q \in (2,\infty)$, be two periodic functions, such that they are even and odd in $z$, respectively. Then, system (\ref{eq8})--(\ref{eq10}), subject to (\ref{eq12})--(\ref{eq14}), has a unique global strong solution $(v,T)$, which is continuously depending on the initial data.
\end{proposition}

In the rest of this section, we always assume that $(v, T)$ is the unique strong solution stated in Proposition \ref{prop2.1}.\\

The following proposition has been proved in \cite{J.Li1}, see Proposition 3.2 there.

\begin{proposition}
  \label{pro3.1}
Suppose the assumptions of Proposition \ref{prop2.1} hold. Then, for any $0<\mathcal{T}<\infty$, we have the following:

(i) Basic energy estimate:
\begin{equation*}
 \sup_{0\leq t\leq\mathcal{T}} \left\|(v,T)\right\|_2^2(t) + \int_0^{\mathcal{T}}\left\|(\nabla_H v, \partial_z T)\right\|_2^2dt\leq Ce^{\mathcal{T}}(\|v_0\|_2^2+\|T_0\|_2^2),
\end{equation*}
where $C$ is a positive constant depending only on $h$;

(ii) $L^{\infty}$ estimate on $T$:
\begin{eqnarray*}
 &&\sup_{0\leq t\leq\mathcal{T}} \|T\|_{\infty}(t) \leq \|T_0\|_{\infty};
\end{eqnarray*}

(iii) $L^r$ growth on $v$:
\begin{eqnarray*}
 &&\sup_{0\leq t\leq \mathcal{T}} \|v\|_{r}(t) \leq C\sqrt{r}, \, \,\,\,\forall r \in [4,\infty),
\end{eqnarray*}
for a positive constant $C$ depending only on $h$, $\mathcal{T}$, and $\|(v_0, T_0)\|_{L^\infty}$.
\end{proposition}

As in \cite{J.Li1}, define
\begin{equation}
   u=\partial_z v,\quad\theta=\nabla_H^{\bot}\cdot v,\quad\eta=\nabla_H\cdot v+\Phi,\label{eq16}
\end{equation}
where $\nabla_H^{\bot}=(-\partial_y,\partial_x)$ and
\begin{equation*}
\Phi=\int_{-h}^zTd\xi-\frac{1}{2h}\int_{-h}^h\int_{-h}^zTd\xi dz.\label{Phi}
\end{equation*}
Then, $u, \theta, \eta$ satisfy (see (3.5)--(3.9) in \cite{J.Li1})
\begin{align}
  \partial_tu &+(v\cdot\nabla_H)u+w\partial_zu-\Delta_H u+ f_0\vec{k}\times u \nonumber\\
&+(u\cdot\nabla_H)v-(\nabla_H\cdot v)u-\nabla_H T=0,\label{eq17}\\
\partial_t\theta &-\Delta_H \theta=-\nabla_H^{\bot}\cdot\Big[(v\cdot \nabla_H)v+w \partial_z v+f_0\vec{k}\times v\Big],\label{eq18}\\
 \partial_t\eta &-\Delta_H \eta=-\nabla_H\cdot\Big[(v\cdot \nabla_H)v+w \partial_z v+f_0\vec{k}\times v\Big]\nonumber\\
&+\partial_z T -w T-\int_{-h}^z \nabla_H\cdot(vT)d\xi+f(x,y,t),\label{eq19}
\end{align}
where
\begin{align*}
  f(x,y,t)=&\frac{1}{2h}\int_{-h}^h\nabla_H\cdot\Big(\nabla_H\cdot(v\otimes v)+f_0\vec{k}\times v\Big)dz\nonumber\\
   &+\frac{1}{2h}\int_{-h}^h\Big(\int_{-h}^z\nabla_H\cdot(vT)d\xi +w T\Big)dz.
\end{align*}

For simplicity of notations, for initial data $(v_0, T_0)$, set
\begin{align}
\|(v_0, T_0)\|_{X_1}:=\|(v_0, T_0)\|_{H^1\cap L^\infty}+\|(\partial_z v_0, \nabla_H T_0)\|_{L^q}+\|v_0\|_{L_z^1(B^1_{q,2}(M))} \label{normX1}
\end{align}
and assume
\begin{equation}\label{M0}
\|(v_0, T_0)\|_{X_1}\leq M_0
\end{equation}
for some positive constant $M_0$.

In the rest of this section, we derive a series of a priori estimates depending only on $q$, $h$, $M_0$, and the positive time $\mathcal{T}$ up to which the estimates are considered. The a priori estimates are carried out in five subsections as follows.

\subsection{A priori $L_t^{\infty}(H^1)$ estimate on $v$}

\begin{proposition}
  \label{pro4.2}
Suppose the assumptions of Proposition \ref{prop2.1} hold. Then, for any positive time $\mathcal T$, it holds that
\begin{align*}
  &&\sup_{0\leq t\leq \mathcal{T}}(\|(\eta, \theta, u)\|_2^2(t)+\|u\|_q^q(t))+\int_{0}^{\mathcal{T}}\left(
  \|\nabla_H(\eta, \theta, u)\|_2^2+\big\||u|^{\frac{q}{2}-1}\nabla_Hu\big\|_2^2\right)dt\leq C
\end{align*}
for a positive constant $C$ depending only on $q$, $h$, $\mathcal{T}$, and $M_0$ (see (\ref{M0})), where $\eta$ and $\theta$ are given by (\ref{eq16}).
\end{proposition}
\begin{proof}
Multiplying \ref{eq17} with $|u|^{q-2}u$, it follows from integration by parts, Proposition \ref{pro3.1}, and the Young inequality that
\begin{eqnarray*}
&&\frac{1}{q}\frac{d}{dt}\|u\|_q^q+\int_{\Omega}|u|^{q-2}\left[|\nabla_Hu|^2+(q-2)\big|\nabla_H|u|\big|^2\right]dxdydz\\
&=&\int_{\Omega} [\nabla_H T+(\nabla_H\cdot v)u-(u\cdot\nabla_H)v]|u|^{q-2}udxdydz\\
&\leq& C\int_{\Omega}\left(|T||u|^{q-2}|\nabla_Hu|+|v||u|^{q-1}|\nabla_Hu|\right)dxdydz\\
&\leq& \frac{1}{2}\int_{\Omega}|u|^{q-2}|\nabla_Hu|^2dxdydz+C\int_{\Omega}\left(|T|^2|u|^{q-2}+|v|^2|u|^q\right)dxdydz\\
&\leq& \frac{1}{2}\int_{\Omega}|u|^{q-2}|\nabla_Hu|^2dxdydz+C(\|v\|_{\infty}^2+1)(\|u\|_q^q+1)
\end{eqnarray*}
and thus
\begin{align}
\frac{d}{dt}\|u\|_q^q+\frac{q}{2}\int_{\Omega}|u|^{q-2}|\nabla_Hu|^2dxdydz\leq C(\|v\|_{\infty}^2+1)(\|u\|_q^q+1). \label{42q}
\end{align}
Similarly,
\begin{equation}
\frac{d}{dt}\|u\|_2^2+\|\nabla_H u\|_{2}^2\leq C(\|v\|_{\infty}^2+1)(\|u\|_2^2+1).\label{ADD24.5.5-2}
\end{equation}

Let $\alpha\geq q$ to be determined.
Multiplying (\ref{42q}) with $\|u\|_q^{\alpha-q}$, it follows from the Young inequality that
 \begin{align*}
\frac{d}{dt}\|u\|_q^{\alpha}\leq C_{\alpha}(\|v\|_{\infty}^2+1)(\|u\|_q^{\alpha}+1), \quad\forall \alpha\geq q.
\end{align*}
Summing this with (\ref{ADD24.5.5-2}) and using the Young inequality yield
\begin{eqnarray}
\frac{d}{dt}\left(\|u\|_q^{\alpha}+\|u\|_q^q+\|u\|_2^2\right)+\big\||u|^{\frac{q}{2}-1}\nabla_Hu\big\|_2^2+\|\nabla_H u\|_{2}^2\nonumber\\
\leq C_\alpha(\|v\|_{\infty}^2+1)(\|u\|_q^{\alpha}+1), \quad \forall \alpha\geq q.\label{42c}
\end{eqnarray}

Multiplying (\ref{eq18}) with $\theta$ (see (\ref{eq16})), it follows from integration by parts that
\begin{align}
\frac{1}{2}\frac{d}{dt}\|\theta\|_2^2+\|\nabla_H \theta\|_2^2=\int_{\Omega}\Big[(v\cdot \nabla_H)v+w \partial_z v+f_0\vec{k}\times v\Big]\cdot\nabla_H^{\bot}\theta dxdydz. \label{eq328-0}
\end{align}
The terms on the right-hand side of (\ref{eq328-0}) are estimated as follows.
Since $\|\nabla_Hv\|_2^2=\|\nabla_H\cdot v\|_2^2+\|\nabla_H^\perp\cdot v\|_2^2$ and by Proposition \ref{pro3.1},
one deduces
\begin{eqnarray}
&&\int_{\Omega}\Big[(v\cdot \nabla_H)v+f_0\vec{k}\times v\Big]\cdot\nabla_H^{\bot}\theta dxdydz\nonumber\\
&\leq& \frac{1}{10}\|\nabla_H \theta\|_2^2 + C(\|v\|_{\infty}^2\|\nabla_H v\|_2^2+\|v\|_2^2)\nonumber\\
&\leq& \frac{1}{10}\|\nabla_H \theta\|_2^2 + C(\|v\|_{\infty}^2+1)(\|(\theta, \eta)\|_2^2+1),\label{eq328-1}
\end{eqnarray}
where the definitions of $\eta$ and $\theta$ (see (\ref{eq16})) were used.
By Proposition \ref{pro3.1} and recalling the definitions of $\eta$ and $\theta$, one deduces by the H\"{o}lder, Minkowski, Ladyzhenskaya, and Young inequalities that
\begin{eqnarray}
&&\int_{\Omega}w \partial_z v\cdot\nabla_H^{\bot}\theta dxdydz \nonumber\\
&=&-\int_{\Omega}\left(\int_{-h}^z\nabla_H\cdot vd\xi\right)u\cdot\nabla_H^{\bot}\theta dxdydz \nonumber\\
& \leq& \int_M\left(\int_{-h}^h |\eta|dz\right) \left(\int_{-h}^h |u||\nabla_H\theta| dz\right)dxdy+C\int_{\Omega}|u||\nabla_H\theta|dxdydz\nonumber\\
& \leq& \int_M\left(\int_{-h}^h|\eta|dz\right)\left(\int_{-h}^h|u|^2dz\right)^{\frac{1}{2}}\left(\int_{-h}^h|\nabla_H\theta|^2dz\right)^{\frac{1}{2}}dxdy
+C\|u\|_2\|\nabla_H\theta\|_2\nonumber\\
& \leq& \left\|\int_{-h}^h |\eta|dz\right\|_{L^{\frac{2q}{q-2}}(M)}\left\|\int_{-h}^h|u|^2dz\right\|_{L^{\frac{q}{2}}(M)}^{\frac{1}{2}}\|\nabla_H\theta\|_2+C\|u\|_2\|\nabla_H\theta\|_2\nonumber\\
& \leq&\left(\int_{-h}^h\|\eta\|_{\frac{2q}{q-2},M}dz\right)\left(\int_{-h}^h\|u\|_{q,M}^2dz\right)^{\frac{1}{2}}\|\nabla_H\theta\|_2+C\|u\|_2\|\nabla_H\theta\|_2\nonumber\\
& \leq& C\left(\int_{-h}^h\|\eta\|_{2,M}^{\frac{q-2}{q}}\|\eta\|_{H^1(M)}^{\frac{2}{q}}dz\right)\|u\|_q\|\nabla_H\theta\|_2+C\|u\|_2\|\nabla_H\theta\|_2\nonumber\\
& \leq& C\|\eta\|_2^{\frac{q-2}{q}}\left(\|\nabla_H \eta\|_2+\|\eta\|_2\right)^{\frac{2}{q}}\|u\|_q\|\nabla_H \theta\|_2+C\|u\|_2\|\nabla_H\theta\|_2\nonumber\\
& \leq& \frac{1}{10}\|\nabla_H(\eta, \theta)\|_2^2+C\|\eta\|_2^2\left(\|u\|_q^{\frac{2q}{q-2}}+1\right)+C\|u\|_2^2\nonumber\\
& \leq& \frac{1}{10}\|\nabla_H(\eta, \theta)\|_2^2+C\left(1+\|\nabla_H v\|_2^2\right)\left(\|u\|_q^{\frac{2q}{q-2}}+\|u\|_2^2+1\right). \label{42d}
\end{eqnarray}
Substituting (\ref{eq328-1}) and (\ref{42d}) into (\ref{eq328-0}) yields
\begin{eqnarray}
\frac{d}{dt}\|\theta\|_2^2+2\|\nabla_H\theta\|_2^2&\leq& \frac{2}{5}\|\nabla_H(\eta, \theta)\|_2^2+C\left(\|v\|_{\infty}^2+\|\nabla_H v\|_2^2+1\right)\nonumber\\
&&\times\left(\|(\theta, \eta)\|_2^2+\|u\|_q^{\frac{2q}{q-2}}+\|u\|_2^2+1\right).\label{eq328-2}
\end{eqnarray}

Multiplying (\ref{eq19}) with $\eta$, it follows from integration by parts that
\begin{align}
\frac{d}{2dt}&\|\eta\|_2^2+\|\nabla_H \eta\|_2^2=\int_{\Omega}\Big[(v\cdot \nabla_H)v+w \partial_z v+f_0\vec{k}\times v\Big]\cdot\nabla_H \eta dxdydz\nonumber\\
& +\int_{\Omega}\left[\left(\partial_z T -w T\right)\eta+\left(\int_{-h}^z vTd\xi\right)\cdot\nabla_H\eta\right] dxdydz=:K_1+K_2.\label{42e}
\end{align}
Estimate for $K_1$ and $K_2$ are given as follows. For $K_1$, similar arguments to (\ref{eq328-1}) and (\ref{42d}) yield
$$
K_1\leq \frac{1}{5}\|\nabla_H (\eta, \theta)\|_2^2 + C\left(\|v\|_{\infty}^2+\|\nabla_H v\|_2^2+1\right)\left(\|(\theta, \eta)\|_2^2+\|u\|_q^{\frac{2q}{q-2}}+\|u\|_2^2+1\right).
$$
For $K_2$, since $\|\nabla_Hv\|_2^2=\|\nabla_H\cdot v\|_2^2+\|\nabla_H^\perp\cdot v\|_2^2$, it follows from the H\"{o}lder and Young
inequalities, the definitions of $\eta$ and $\theta$, and Proposition \ref{pro3.1} that
\begin{eqnarray*}
K_2&\leq& \left(\|\partial_z T\|_2+\|T\|_{\infty}\|w\|_2\right)\|\eta\|_2+C\|T\|_\infty\|v\|_2\|\nabla_H \eta\|_2\\
&\leq &\frac{1}{10}\|\nabla_H(\eta,\theta)\|_2^2+C(\|\partial_z T\|_2^2+\|\eta\|_2^2+\|\nabla_H v\|_2^2+\|v\|_2^2)\\
&\leq &\frac{1}{10}\|\nabla_H(\eta,\theta)\|_2^2+C(\|\partial_z T\|_2^2+\|(\theta,\eta)\|_2^2+1).
\end{eqnarray*}
Substituting the estimates for $K_1$ and $K_2$ into (\ref{42e}) yields
\begin{eqnarray}
\frac{d}{dt}\|\eta\|_2^2+2\|\nabla_H \eta\|_2^2&\leq& \frac{3}{5}\|\nabla_H(\eta, \theta)\|_2^2+C\|\partial_z T\|_2^2+C\left(\|v\|_{\infty}^2+\|\nabla_H v\|_2^2+1\right)\nonumber\\
& &\times\left(\|(\theta, \eta)\|_2^2+\|u\|_q^{\frac{2q}{q-2}}+\|u\|_2^2+1\right).\label{eq328-3}
\end{eqnarray}

Choosing $\alpha=\max\left\{\frac{2q}{q-2},q\right\}$ in (\ref{42c}) and summing up (\ref{42c}), (\ref{eq328-2}), and (\ref{eq328-3}), one obtains by the Young inequality that
\begin{eqnarray}
&& \frac{d}{dt}\left(\|(\eta,\theta, u)\|_2^2+\|u\|_q^{\alpha}+\|u\|_q^q\right)+\|\nabla_H(u, \eta, \theta)\|_2^2+\big\||u|^{\frac{q}{2}-1}\nabla_Hu\big\|_2^2\nonumber\\
&\leq&  C(\|v\|_{\infty}^2+\|\nabla_H v\|_2^2+1)\left(\|u\|_q^\alpha+\|u\|_q^{\frac{2q}{q-2}}+\|u\|_2^2+\|(\eta,\theta)\|_2^2+1\right)+C\|\partial_z T\|_2^2\nonumber\\
&\leq&  C(\|v\|_{\infty}^2+\|\nabla_H v\|_2^2+1)\left(\|u\|_q^\alpha+\|(\eta,\theta)\|_2^2+1\right)+C\|\partial_z T\|_2^2.\label{ADD24.5.5-1}
\end{eqnarray}

Denote
$$A_1=\|(\eta,\theta, u)\|_2^2+\|u\|_q^{\alpha}+\|u\|_q^q+e,\quad B_1=\|\nabla_H(u, \eta, \theta)\|_2^2+\big\||u|^{\frac{q}{2}-1}\nabla_Hu\big\|_2^2.$$
Thanks to Lemma \ref{lemma29} and Proposition \ref{pro3.1}, we have
\begin{eqnarray*}
\|v\|_\infty& \leq& C\max\left\{1,\sup_{r\geq1}\frac{\|v\|_r}{\sqrt{r}}\right\}\log^{\frac{1}{2}}\left(\|\nabla_H v\|_4+\|v\|_4+\|u\|_q+\|v\|_q+e\right)\\
& \leq&C\log^{\frac{1}{2}}\left(\|\nabla_H\cdot v\|_4+\|\nabla_H^{\bot}\cdot v\|_4+\|u\|_q+e\right)\\
& \leq& C\log^{\frac{1}{2}}\left(\|\eta\|_4+\|\theta\|_4+\|u\|_q+e\right)\\
& \leq&C\log^{\frac{1}{2}}\left(\|(\eta, \theta)\|_2+\|\nabla_H(\eta, \theta)\|_2+\|\partial_z(\eta, \theta)\|_2+\|u\|_q+e\right)\\
& \leq&C\log^{\frac{1}{2}}\left(\|\nabla_H(\eta, \theta, u) \|_2+\|(\eta,\theta)\|_2+\|u\|_q+e\right)\\
& \leq&C\log^{\frac{1}{2}}\left(A_1+B_1+e\right).
\end{eqnarray*}
Therefore, by (\ref{ADD24.5.5-1}), one has
\begin{eqnarray*}
\frac{d}{dt}A_1+B_1\leq C\Big(\|\nabla_H v\|_2^2+1+\log\left(A_1+B_1+e\right)\Big)A_1+C\|\partial_z T\|_2^2,
\end{eqnarray*}
from which, by Lemma \ref{lemma5}, the conclusion follows.
\end{proof}

As a corollary of Proposition \ref{pro4.2}, one has the following corollary.

\begin{corollary}\label{COR1-1}
Suppose the assumptions of Proposition \ref{prop2.1} hold and $\mathcal T$ is any given positive time. Then, for any $t\in(0,\mathcal T)$, it holds that
\begin{eqnarray*}
\|\nabla_Hv\|_2(t)&\leq&C,\\
\left(\left\|\int_{-h}^h|\nabla_H v|dz\right\|_{L^{2q}(M)}^2+\int_{-h}^h\left\|\nabla_H v\right\|_{2q,M}^2dz\right)(t) &\leq&
C \left(\|\nabla_H(\eta, \theta)\|_2^{2-\frac2q}+1\right)(t),\\
\left(\left\|\int_{-h}^h|u|dz\right\|_{L^{2q}(M)}^2+\int_{-h}^h\left\|u\right\|_{2q,M}^2dz\right)(t) &\leq& C\left(1+\big\|\nabla_H|u|^{\frac{q}{2}}\big\|_2^{\frac{2}{q}}\right)(t),
\end{eqnarray*}
where $C$ is a positive constant $C$ depending only on $q$, $h$, $M_0$, and $\mathcal{T}$.
\end{corollary}

\begin{proof}
The first conclusion is a direct corollary of Propositions \ref{pro3.1}--\ref{pro4.2} by noticing that $\|\nabla_Hv\|_2^2=\|\nabla_H\cdot v\|_2^2+\|\nabla_H^\perp\cdot v\|_2^2$.
By the elliptic estimate, it follows from the Gagliardo-Nirenberg and H\"{o}lder inequalities, the definitions of $\eta, \theta, \Phi$, and Propositions \ref{pro3.1}--\ref{pro4.2} that
\begin{eqnarray*}
 \int_{-h}^h\left\|\nabla_H v\right\|_{2q,M}^2dz
& \leq&C\int_{-h}^h\left\| (\nabla_H\cdot v,\nabla_H^\perp\cdot v)\right\|_{2q,M}^2dz\\
&\leq&\int_{-h}^h\|(\eta, \Phi, \theta)\|_{2q,M}^2dz\\
& \leq& C\int_{-h}^h \|(\eta, \theta)\|_{2,M}^{\frac2q} \|(\eta, \theta,\nabla_H \eta, \nabla_H\theta)\|_{2,M} ^{2-\frac2q} dz+C\\
& \leq&C\left(\|\nabla_H(\eta, \theta)\|_2^{2-\frac2q}+1\right)
\end{eqnarray*}
and
\begin{eqnarray*}
\int_{-h}^h\left\|u\right\|_{2q,M}^2dz&=& \int_{-h}^h \big\||u|^{\frac{q}{2}}\big\|_{4,M}^{\frac{4}{q}}dz\\
& \leq& C\int_{-h}^h \big\||u|^{\frac{q}{2}}\big\|_{2,M}^{\frac{2}{q}}\left(\big\|\nabla_H|u|^{\frac{q}{2}}\big\|_{2,M}+\big\||u|^{\frac{q}{2}}\big\|_{2,M}\right)^{\frac{2}{q}}dz\\
& \leq&C\left(1+\big\|\nabla_H|u|^{\frac{q}{2}}\big\|_2^{\frac{2}{q}}\right).
\end{eqnarray*}
The other estimates follow from the above two estimates by further using the H\"{o}lder and Minkowski inequalities.
\end{proof}

\subsection{Anisotropic parabolic estimates on $v, \eta, \theta$}

For simplifying notations, throughout this subsection, we use $L_{(0,t)}^m$, $L^m_z$, and $L^m_M$, respectively, to denote the $L^m$ norms in the time variable, $z$ variable, and $(x,y)$ variables, over $(0, t)$, $(-h,h),$ and $M$.

\begin{proposition}
\label{PRO-PARA-EST1}
Suppose the assumptions of Proposition \ref{prop2.1} hold and $\mathcal T$ is any given positive time. Then, for any $t\in(0,\mathcal T)$, it holds that
  \begin{eqnarray*}
 \|\nabla_H^2 v\|_{L_z^1L^2_{(0,t)}L_M^q}^2\leq C+C\int_0^t\|\nabla_HT\|_q^2ds \label{eq34}
\end{eqnarray*}
for a positive constant $C$ depending only on $q, h, M_0$, and $\mathcal T$.
\end{proposition}

\begin{proof}
Applying the parabolic estimate to equation (\ref{eq8}) yields
\begin{eqnarray}
&\|\nabla_H^2v(.,z,.)\|_{L^2\left(0,t; L^q(M)\right)}
\leq C\Big(\|(-(v\cdot\nabla_H)v-w u-\nabla_H p_s
+\int_{-h}^z\nabla_HTd\xi\nonumber \\&-f_0\vec{k}\times v)(.,z,.)\|_{L^2\left(0,t; L^q(M)\right)}+\|v_0(.,z)\|_{B_{q,2}^1(M)}\Big), \quad \forall z\in (-h, h), \nonumber
\end{eqnarray}
and thus by Proposition \ref{pro3.1}
\begin{eqnarray}
\|\nabla_H^2 v\|_{L_z^1L^2_{(0,t)}L_M^q}^2
 &\leq&  C\left(1+\big\||v|\nabla_Hv\big\|_{L_z^1L_{(0,t)}^2L_M^q}+\|w u\|_{L_z^1L_{(0,t)}^2L_M^q}\right.\nonumber\\
 &&+\|\nabla_H p_s\|_{L_{(0,t)}^2L_M^q}+\left.\|\nabla_H T\|_{L_z^1L_{(0,t)}^2L_M^q}\right)^2.\label{eq33}
\end{eqnarray}

The terms on the right-hand side of (\ref{eq33}) are estimated as follows.
First, as $p_s$ satisfies (\ref{eq15}), by the elliptic estimates, Proposition \ref{pro3.1}, and Minkowski inequality, one has
\begin{align}
\|\nabla_H p_s\|_{L^2_{(0,t)}L_M^q}
&\leq C\big\||v|\nabla_H v\big\|_{L_z^1L_{(0,t)}^2L_M^q}+\left\|v\right\|_{L_z^1L_{(0,t)}^2L_M^q}+\left\|\nabla_H T\right\|_{L_z^1L_{(0,t)}^2L_M^q}\nonumber\\
&\leq C\| \nabla_H T \|_{L^2(0,t;L^q(\Omega))}+C\big\||v|\nabla_Hv\big\|_{L_z^1L_{(0,t)}^2L_M^q}+C\nonumber\\
& = C\left(\int_0^t \|\nabla_H T\|_q^2 ds\right)^{\frac{1}{2}}+C\big\||v|\nabla_Hv\big\|_{L_z^1L_{(0,t)}^2L_M^q}+C. \label{a4}
\end{align}

Next, it follows from the H\"{o}lder, Gagliardo-Nirenberg, and Young inequalities that
\begin{eqnarray*}
\big\||v|\nabla_H v\big\|_{L_z^1L_{(0,t)}^2L_M^q}
&=& \int_{-h}^h\left(\int_0^t\left(\int_M|v|| \nabla_Hv|^qdxdy\right)^\frac{2}{q}ds\right)^{\frac{1}{2}}dz \nonumber\\
&\leq& \,\int_{-h}^h\left(\int_0^t\|v \|_{\infty,M}^2\|\nabla_Hv \|_{q,M}^2ds\right)^{\frac{1}{2}}dz\nonumber\\
&\leq& C\int_{-h}^h\left[\int_0^t\left(\|v \|_{4,M}^{\frac{8(q-1)}{5q-4}}\|\nabla_H^2v \|_{q,M}^{\frac{2q}{5q-4}}
 +\|v \|_{4,M}^2\right)\right.\nonumber\\
&& \times\left.\left(\|v \|_{4,M}^{\frac{4q}{5q-4}}\|\nabla_H^2v \|_{q,M}^{\frac{6q-8}{5q-4}}
 +\|v \|_{4,M}^2\right)ds\right]^{\frac{1}{2}}dz
 \end{eqnarray*}
from which, noticing that Propositions \ref{pro3.1}--\ref{pro4.2} and Lemma \ref{pro3.2} imply
$$\sup_{0\leq s\leq t}\sup_{-h\leq z\leq h}\|v(.,z,s)\|_{4,M}\leq C,$$ one obtains by the Young inequality that
 \begin{align}
 \big\||v|\nabla_H v\big\|_{L_z^1L_{(0,t)}^2L_M^q}
 \leq&C\int_{-h}^h\left(\int_0^t\|\nabla_H^2v\|_{q,M}^{\frac{8q-8}{5q-4}}ds\right)^\frac12dz+C\nonumber\\
\leq& \delta\int_{-h}^h\left(\int_0^t\|\nabla_H^2v(.,z,s)\|_{q,M}^{2}ds\right)^{\frac{1}{2}}dz +C_{\delta}\nonumber\\
=&\delta\|\nabla_H^2 v\|_{L_z^1L_{(0,t)}^2L_M^q}+C_{\delta},\quad\forall\delta>0. \label{a6}
\end{align}

Finally, we estimate $\|w u\|_{L_z^1L_{(0,t)}^2L_M^q}$. Noticing that from (\ref{eq11}) $\int_M w dxdy=0$, it follows from the Gagliardo-Nirenberg and Poincar$\acute{e}$ inequalities that
\begin{eqnarray*}
\|w\|_{L^{2q}(M)}^2&\leq& C\|w\|_{L^2(M)}\|\nabla_Hw\|_{L^q(M)}\nonumber\\
&\leq& C\left\|\int_{-h}^h|\nabla_H v|dz\right\|_{L^2(M)}\left\|\int_{-h}^h|\nabla_H^2v|dz\right\|_{L^q(M)}\nonumber\\
&\leq& C\|\nabla_Hv\|_2\|\nabla_H^2v\|_{L^q_ML^1_z}\label{eqH2Lq2}
\end{eqnarray*}
and thus, by Corollary \ref{COR1-1}, it holds that
$\|w\|_{L^{2q}(M)}^2\leq C \|\nabla_H^2v\|_{L^q_ML^1_z}.$
By virtue of this estimate, it follows from the H\"{o}lder and Ladyzhenskaya inequalities that
\begin{eqnarray*}
\|w u\|_{L_z^1L_{(0,t)}^2L_M^q}
&\leq&\int_{-h}^h\left(\int_0^t\|w\|_{2q,M}^{2}\big\||u|^{\frac{q}{2}}\big\|_{4,M}^{\frac{4}{q}}ds\right)^{\frac{1}{2}}dz\nonumber\\
&\leq& C\int_{-h}^h\left(\int_0^t\|\nabla_H^2v\|_{L^q_ML^1_z}\big\||u|^{\frac{q}{2}}\big\|_{2,M}^{\frac{2}{q}} \big\||u|^{\frac{q}{2}}\big\|_{H^1(M)}^{\frac{2}{q}} ds\right)^{\frac{1}{2}}dz\nonumber\\
&\leq& C\left(\int_{-h}^h\int_0^t\|\nabla_H^2v\|_{L^q_ML^1_z}\big\||u|^{\frac{q}{2}}\big\|_{2,M}^{\frac{2}{q}} \big\||u|^{\frac{q}{2}}\big\|_{H^1(M)}^{\frac{2}{q}} dsdz\right)^{\frac{1}{2}}\nonumber\\
&\leq& C\left[\int_0^t\|\nabla_H^2v\|_{L^q_ML^1_z}\int_{-h}^h\left(\|u\|_{q,M}^2+\|u\|_{q,M}\big\|\nabla_H|u|^{\frac{q}{2}}\big\|_{2,M}^{\frac{2}{q}}\right)dz ds\right]^{\frac{1}{2}}\nonumber\\
&\leq& C\left[\int_0^t\|\nabla_H^2v\|_{L^q_ML^1_z} \left(\|u\|_{q}^2+\|u\|_{q}\big\|\nabla_H|u|^{\frac{q}{2}}\big\|_{2}^{\frac{2}{q}}\right) ds\right]^{\frac{1}{2}}\nonumber\\
&\leq& t^{\frac{q-2}{4q}}C\left(\sup_{0\leq s\leq t}\|u\|_q\right)\left(\int_0^t \big\|\nabla_H |u|^{\frac{q}{2}}\big\|_2^{2}ds\right)^{\frac{1}{2q}}\|\nabla_H^2v\|_{L_{(0,t)}^2L_M^qL_z^1}^{\frac{1}{2}}\nonumber\\
&&  +C t^{\frac{1}{4}}\left(\sup_{0\leq s\leq t}\|u\|_q^2\right)\|\nabla_H^2v\|_{L_{(0,t)}^2L_M^qL_z^1}^{\frac{1}{2}}\label{eqH2Lq3}
\end{eqnarray*}
from which, by Proposition \ref{pro4.2}, and since $t\in(0,\mathcal T)$, one gets by means of the Minkowski and Young inequalities
\begin{align}
\|w u\|_{L_z^1L_{(0,t)}^2L_M^q}&\leq C\|\nabla_H^2v\|_{L_{(0,t)}^2L_M^qL_z^1}^{\frac{1}{2}} \leq C\|\nabla_H^2v\|_{L_z^1L_{(0,t)}^2L_M^q}^{\frac{1}{2}} \nonumber\\
&\leq\delta\|\nabla_H^2 v\|_{L_z^1L_{(0,t)}^2L_M^q}+C_{\delta},\quad\forall\delta>0.  \label{a8}
\end{align}

Substituting (\ref{a4}), (\ref{a6}), (\ref{a8}) into (\ref{eq33}), and choosing $\delta$ sufficiently small lead to the conclusion.
\end{proof}

\begin{proposition}
  \label{pro4.3}
Suppose the assumptions of Proposition \ref{prop2.1} hold and $\underline{\mathcal T}<\mathcal T$ is given. Then, the following estimate holds
\begin{eqnarray*}
&&\|\nabla_H (\eta,\theta)\|_{L^2\left(\underline{\mathcal T}, \mathcal{T}; L_z^2L^q_M\right)}\leq C
\end{eqnarray*}
for a positive constant $C$ depending only on $q$, $h$, $M_0$, $\underline{\mathcal T}$, and $\mathcal{T}$, where $\eta$ and $\theta$ are given by (\ref{eq16}).
\end{proposition}

\begin{proof}
By Proposition \ref{pro4.2}, one can choose a positive time $\mathcal T_\sharp\in\left(0.5\underline{\mathcal T},\underline{\mathcal T}\right)$, such that
\begin{equation}\label{ADD24.5.5-3}
\|\eta\|_2(\mathcal T_\sharp)+\|\nabla_H\eta\|_2(\mathcal T_\sharp)\leq C.
\end{equation}

Rewrite (\ref{eq19}) as
\begin{align*}
 \partial_t\eta -\Delta_H \eta=-\nabla_H\cdot\left(G-\frac{1}{2h}\int_{-h}^h G d\xi\right)+
 \partial_z T -w T+\frac{1}{2h}\int_{-h}^h w Td\xi,
\end{align*}
where
\begin{equation}\label{EQG}
G=(v\cdot \nabla_H)v+w \partial_z v+f_0\vec{k}\times v-\int_{-h}^z vTd\xi.
\end{equation}
Decompose $\eta$ as $$\eta=\eta_1+\eta_2,$$ where $\eta_1$ and $\eta_2$, respectively, satisfy
\begin{eqnarray}
\label{A}
 \left\{
  \begin{array}{l}
     \partial_t\eta_1-\Delta_H\eta_1
    = -\nabla_H\cdot\left(G-\frac{1}{2h}\int_{-h}^h G d\xi\right),\\
\left.\eta_1\right|_{t=\mathcal T_\sharp}=0,
  \end{array}
  \right.
\end{eqnarray}
and
\begin{eqnarray}
\label{B}
 \left\{
  \begin{array}{l}
 \partial_t\eta_2-\Delta_H\eta_2
    = \partial_z T -w T+\frac{1}{2h}\int_{-h}^h w Td\xi, \\
 \left.\eta_2\right|_{t=\mathcal{T}_\sharp}=\left.\eta\right|_{t=\mathcal T_\sharp}.
  \end{array}
  \right.
\end{eqnarray}

Applying the parabolic estimate to (\ref{B}) yields
\begin{align*}
&\left\|\eta_2(., z, .)\right\|_{L^2(\mathcal T_\sharp,\mathcal{T};H^2(M))}\\
\leq& C\left\|\partial_z T -w T+\frac{1}{2h}\int_{-h}^h w Td\xi\right\|_{L^2\left(\mathcal T_\sharp,\mathcal{T};L^2(M)\right)}+C\left\|\left.(\eta_2)\right|_{t=\mathcal T_\sharp}\right\|_{H^1(M)}
\end{align*}
and thus, by the Sobolev embedding, (\ref{ADD24.5.5-3}), and Propositions \ref{pro3.1}--\ref{pro4.2}, it follows
\begin{align}
\left\|\nabla_H\eta_2(., z, .) \right\|_{L_z^2L^2_{(\mathcal T_\sharp,\mathcal{T})}L^q_M}
\leq& C\left\|\eta_2(., z, .) \right\|_{L_z^2L^2_{(\mathcal T_\sharp,\mathcal{T})}H^2_M}
\leq C.\label{eq3.81}
\end{align}
Applying the parabolic estimate to (\ref{A}) yields
\begin{align*}
&\left\|\nabla\eta_1(., z, .)\right\|_{L^2(\mathcal T_\sharp,\mathcal{T};L^q_M)}
\leq C \left\|G\right\|_{L^2(\mathcal T_\sharp,\mathcal{T};L^q_M)}+C\left\|\int_{-h}^hGdz\right\|_{L^2(\mathcal T_\sharp,\mathcal{T};L^q_M)}
\end{align*}
and thus
\begin{align}
\left\|\nabla\eta_1\right\|_{L^2(\mathcal T_\sharp,\mathcal{T};L_z^2L^q_M)}
\leq C\left\|G\right\|_{L^2(\mathcal T_\sharp,\mathcal{T};L^2_zL^q_M)}.\label{eq3.82}
\end{align}

Noticing that $|v|\leq \frac{1}{2h}\int_{-h}^h |v|dz+\int_{-h}^h |u|dz$, one deduces
by Proposition \ref{pro3.1}, Corollary \ref{COR1-1}, and the Sobolev and H\"{o}lder inequalities that
\begin{eqnarray*}
&&\int_{-h}^h\big\||v||\nabla_Hv|+|w|| \partial_z v|\big\|_{q,M}^2dz\\
&\leq& C\int_{-h}^h\left\| \left(\int_{-h}^h(|v|+|u|)dz\right)|\nabla_H v|+\left(\int_{-h}^h|\nabla_H v|dz\right)|u|\right\|_{q,M}^2dz\\
&\leq &C\left(\left\|\int_{-h}^h|v|dz\right\|_{L^{2q}(M)}^2+\left\|\int_{-h}^h|u|dz\right\|_{L^{2q}(M)}^2\right)\int_{-h}^h\left\|\nabla_H v\right\|_{2q,M}^2dz\\
&&+C\left\|\int_{-h}^h|\nabla_H v|dz\right\|_{L^{2q}(M)}^2\int_{-h}^h\left\|u\right\|_{2q,M}^2dz\\
&\leq & C\Big(\big\|\nabla_H|u|^{\frac{q}{2}}\big\|_2^{\frac{2}{q}}+1\Big)\left(\|\nabla_H(\eta, \theta)\|_2^{2-\frac{2}{q}}+1\right),
\end{eqnarray*}
from which,by the Young inequality and Proposition \ref{pro4.2}, one gets
\begin{eqnarray}
\big\||v||\nabla_Hv|+|w|| \partial_z v|\big\|_{L^2\left(\mathcal T_\sharp,\mathcal{T}; L_z^2L^q_M\right)}^2&\leq& C\int_{\mathcal T_\sharp}^{\mathcal{T}}\left(\|\nabla_H(\eta, \theta)\|_2^2+\|\nabla_H|u|^{\frac{q}{2}}\|_2^2+1\right)dt\nonumber\\
&\leq& C.
\label{eq3.83}
\end{eqnarray}
By Propositions \ref{pro3.1}--\ref{pro4.2}, it follows from the H\"{o}lder inequality that
\begin{align}
&\big\| f_0\vec{k}\times v \big\|_{L^2(\mathcal T_\sharp,\mathcal{T};L_z^2L^q_M)}^2\leq C \label{eq3.85}
\end{align}
and
\begin{align}
&\left\|\int_{-h}^{z}vTd\xi\right\|_{L^2(\mathcal T_\sharp,\mathcal{T};L^2_zL^q_M)}^2\leq C\|T\|^2_{L^\infty\left(\Omega\times\left(\mathcal T_\sharp, \mathcal{T}\right)\right)}\|v\|_{L^2(\mathcal T_\sharp,\mathcal{T};L^q(\Omega))}^2 \leq C .\label{eq3.86}
\end{align}

Recalling (\ref{EQG}), one deduces by (\ref{eq3.83})--(\ref{eq3.86}) that
$\left\|G\right\|_{L^2(\mathcal T_\sharp,\mathcal{T};L_z^2L^q_M)}\leq C.$ With the aid of this, it follows from (\ref{eq3.82}) that
$\left\|\nabla\eta_1(., z, .)\right\|_{L^2(\mathcal T_\sharp,\mathcal{T};L^2_zL^q_M)} \leq C,$
which combined with (\ref{eq3.81}) leads to the conclusion.
\end{proof}

\subsection{Short in time a priori estimate on $\nabla_HT$}

\begin{proposition}
  \label{pro3.6}
Suppose the assumptions of Proposition \ref{prop2.1} hold. Then, there exists a positive time $\mathcal T_0$ depending only on $h$ and $M_0$, such that
$$
\sup_{0\leq t \leq \mathcal T_0} \|\nabla_HT\|_q^q(t) + \int_0^{\mathcal T_0}\int_{\Omega}|\nabla_HT|^{q-2}|\partial_z\nabla_HT|^2 dxdydzdt\leq C
$$
for a positive constant $C$ depending only on $h$ and $M_0$.
\end{proposition}

\begin{proof}
Multiplying (\ref{eq10}) by $-\,\text{div}_H (|\nabla_HT|^{q-2}\nabla_HT)$ and integrating over $\Omega$, it follows from integration by parts that
\begin{eqnarray}
 &&\frac{1}{q}\frac{d}{dt}\|\nabla_HT\|_q^q+\int_{\Omega}|\nabla_HT|^{q-2}(|\partial_z\nabla_HT|^2+(q-2)|\partial_z|\nabla_HT||^2)dxdydz \nonumber\\
 &=&-\int_{\Omega}|\nabla_HT|^{q-2}(\nabla_H T\cdot \nabla_H v+\partial_zT\nabla_Hw)\cdot \nabla_HTdxdydz\nonumber \\
 &=&-\int_{\Omega}|\nabla_HT|^{q-2}\nabla_H T\cdot \nabla_H v\cdot\nabla_HTdxdydz+\int_{\Omega}T|\nabla_HT|^{q-2}\nabla_H\partial_zw\cdot\nabla_HTdxdydz\nonumber \\
 &&+\int_{\Omega}T\nabla_Hw\cdot
 \partial_z(|\nabla_HT|^{q-2}\nabla_HT)dxdydz \nonumber \\
 &\leq&C\int_M\left(\int_{-h}^h|\nabla_H^2 v|dz\right)\left(\int_{-h}^{h}|\nabla_HT|^{q-2}|\partial_z\nabla_HT|dz\right)dxdy \nonumber\\
 &&+\int_{\Omega}(|\nabla_Hv||\nabla_HT|^{q}+C |\nabla_H^2 v||\nabla_HT|^{q-1}) dxdydz. \label{ADD24.5.5-7}
\end{eqnarray}

Noticing that $T|_{z=0}=0$, one has
\begin{eqnarray}
&&\int_{\Omega}|\nabla_H^2 v||\nabla_HT|^{q-1}dxdydz\nonumber\\
&\leq&(q-1)\int_M\left(\int_{-h}^h|\nabla_HT|^{q-2}|\partial_z\nabla_HT|dz\right)\left(\int_{-h}^h|\nabla_H^2v|dz\right)dxdy.
\label{ADD24.5.5-8}
\end{eqnarray}
By the H\"{o}lder, Young, and Minkowski inequalities, one has
\begin{eqnarray}
&&\int_M\left(\int_{-h}^h|\nabla_H^2 v|dz\right)\left(\int_{-h}^{h}|\nabla_HT|^{q-2}|\partial_z\nabla_HT|dz\right)dxdy \nonumber \\
&\leq& \int_M\left(\int_{-h}^h|\nabla_H^2v|dz\right)\left(\int_{-h}^{h}|\nabla_HT|^{q-2}dz\right)^{\frac{1}{2}} \left(\int_{-h}^{h}|\nabla_HT|^{q-2}|\partial_z\nabla_H T|^2dz\right)^{\frac{1}{2}}dxdy \nonumber \\
&\leq& \left[\int_M\left(\int_{-h}^h|\nabla_H^2v|dz\right)^qdxdy\right]^{\frac{1}{q}}\|\nabla_H T\|_q^{\frac{q-2}{2}}\big\||\nabla_HT|^{\frac{q}{2}-1}\partial_z\nabla_HT\big\|_2 \nonumber \\
&\leq& \frac{1}{4}\||\nabla_HT|^{\frac{q}{2}-1}|\partial_z\nabla_HT|\|_2^2+C\|\nabla_H^2v\|_{L^q_M L^1_z}^2\|\nabla_H T\|_q^{q-2}\nonumber\\
&\leq& \frac{1}{4}\||\nabla_HT|^{\frac{q}{2}-1}|\partial_z\nabla_HT|\|_2^2+C\|\nabla_H^2v\|_{L^1_z L^q_M }^2\|\nabla_H T\|_q^{q-2}.\label{ESTJ1}
\end{eqnarray}
Substituting (\ref{ADD24.5.5-8}) and (\ref{ESTJ1}) into (\ref{ADD24.5.5-7}) yields
\begin{eqnarray}
 &&\frac{1}{q}\frac{d}{dt}\|\nabla_HT\|_q^q+\frac34\int_{\Omega}|\nabla_HT|^{q-2} |\partial_z\nabla_HT|^2 dxdydz \nonumber\\
 &\leq& \int_{\Omega}|\nabla_Hv||\nabla_HT|^{q}dxdydz+C\|\nabla_H^2v\|_{L^1_z L^q_M }^2\|\nabla_H T\|_q^{q-2}. \label{ADD24.5.5-9}
\end{eqnarray}

Due to $T|_{z=0}=0$, it follows from the H\"older and Young inequalities that
\begin{eqnarray}
&&\int_{\Omega}|\nabla_Hv||\nabla_HT|^{q}dxdydz\nonumber\\
 &\leq& q\int_M\left(\int_{-h}^h |\nabla_HT|^{q-1}|\partial_z\nabla_H T|dz\right)\left(\int_{-h}^h|\nabla_Hv|dz\right)dxdy \nonumber\\
 &\leq&C\left\|\int_{-h}^h|\nabla_Hv|dz\right\|_{L^\infty(M)}\|\nabla_HT\|_q^{\frac{q}{2}}\big\||\nabla_H T|^{\frac{q-2}{2}}\partial_z\nabla_HT\big\|_2\nonumber\\
 &\leq&\frac{1}{4}\big\||\nabla_H T|^{\frac{q-2}{2}}\partial_z\nabla_HT\big\|_2^2+C\|\nabla_Hv\|_{L^1_zL^\infty_M}^2\|\nabla_HT\|_q^q,
 \label{ADD24.5.5-10}
 \end{eqnarray}
 from which, noticing that $\|\nabla_Hv(\cdot,z,t)\|_{L^\infty(M)}\leq C\|\nabla_H^2v(\cdot,z,t)\|_{L^q(M)}$ guaranteed by the Sobolev and Poincare inequalities, one obtains
 \begin{equation*}
 \int_{\Omega}|\nabla_Hv||\nabla_HT|^{q}dxdydz\leq  \frac{1}{4}\big\||\nabla_H T|^{\frac{q-2}{2}}\partial_z\nabla_HT\big\|_2^2
 +C\|\nabla_H^2v\|_{L^1_zL^q_M}^{2}\|\nabla_HT\|_q^q.\label{ADD24.5.5-11}
\end{equation*}
Thanks to the above, it follows from (\ref{ADD24.5.5-9}) that
\begin{eqnarray}
\frac{d}{dt}\|\nabla_HT\|_q^q+\int_{\Omega}|\nabla_HT|^{q-2}|\partial_z\nabla_HT|^2dxdydz
\leq C\|\nabla_H^2v\|_{L^1_zL^q_M}^2(\|\nabla_H T\|_q^{q}+1),\label{ADD5.4-1}
\end{eqnarray}
where the Young inequality was used.

By Proposition \ref{PRO-PARA-EST1}, it follows from the Minkowski inequality that
$$
\int_0^t \|\nabla_H^2v\|_{L^1_zL^q_M}^2 ds=\|\nabla_H^2 v\|_{L^2_{(0,t)}L_z^1L_M^q}^2\leq
 \|\nabla_H^2 v\|_{L_z^1L^2_{(0,t)}L_M^q}^2\leq C+C\int_0^t\|\nabla_HT\|_q^2ds.
$$
With the aid of this and applying the Gronwall inequality to (\ref{ADD5.4-1}) yield 
\begin{eqnarray}
&&\|\nabla_HT\|_q^q(t)+\int_0^t\left\||\nabla_HT|^{\frac  q2-1}\nabla_H\partial_zT\right\|_2^2ds\nonumber\\
&\leq&
e^{C\int_0^t \|\nabla_H^2v\|_{L^1_zL^q_M}^2 ds}(\|\nabla_HT_0\|_q^q+1)
\leq Ce^{C+C\int_0^t\|\nabla_HT\|_q^2ds},\label{0813-1}
\end{eqnarray}
from which, setting $F(t):=\int_0^t\|\nabla_HT\|_q^2ds$, one arrives at 
$$
F'(t)=\|\nabla_HT\|_q^2(t)\leq C_*e^{C_*F(t)}
$$
for a positive constant depending only on $h$ and $M_0$. Solving the above ODE yields  
$F(t)\leq\frac{\log 2}{C_*}$ for any $t\leq\mathcal T_0:=\frac{1}{2C_*}$. Thanks to this and recalling the definition of $F(t)$, the conclusion follows from (\ref{0813-1}).
\end{proof}

\subsection{Global in time a priori estimate on $\nabla_HT$}

\begin{proposition}
  \label{pro4.4}
Suppose the assumptions of Proposition \ref{prop2.1} hold and $\mathcal T$ is a given positive time. Then, the following estimate holds
\begin{eqnarray*}
\sup_{0\leq t \leq \mathcal{T}}\|\nabla_H T\|_q^q + \int_{0}^{\mathcal{T}}\int_{\Omega}|\nabla_HT|^{q-2}|\partial_z\nabla_HT|^2dxdydzdt\leq C
\end{eqnarray*}
for a positive constant $C$ depending only on $q$, $h$, $M_0$, and $\mathcal{T}$.
\end{proposition}
\begin{proof}
Let $\mathcal T_0$ be the positive time stated in Proposition \ref{pro3.6}.
Due to Proposition \ref{pro3.6}, it suffices to prove
\begin{eqnarray*}
\sup_{\mathcal{T}_0\leq t \leq \mathcal{T}}\|\nabla_H T\|_q^q + \int_{\mathcal{T}_0}^{\mathcal{T}}\int_{\Omega}|\nabla_HT|^{q-2}|\partial_z\nabla_HT|^2dxdydzdt\leq C.
\end{eqnarray*}
Recall (\ref{ADD24.5.5-9}), that is
\begin{eqnarray}
 &&\frac{1}{q}\frac{d}{dt}\|\nabla_HT\|_q^q+\frac34\int_{\Omega}|\nabla_HT|^{q-2} |\partial_z\nabla_HT|^2 dxdydz \nonumber\\
 &\leq & \int_{\Omega}|\nabla_Hv||\nabla_HT|^{q}dxdydz+C\|\nabla_H^2v\|_{L^1_z L^q_M }^2\|\nabla_H T\|_q^{q-2}.
\label{eq42}
\end{eqnarray}

Decompose $v=v_T+v_s+\frac{1}{|M|}\int_Mvdxdy$, where $v_T$ and $v_s$, respectively, satisfy the elliptic systems
\begin{eqnarray}
\label{eq41}
 \left\{
  \begin{array}{l}
    \nabla_H\cdot v_T=-\Phi+\frac{1}{|M|}\int_M\Phi dxdy,\\
    \nabla_H^{\bot}\cdot v_T=0,\quad\int_Mv_Tdxdy=0,\quad v_T\text{ is periodic in }x,y,
  \end{array}
  \right.
\end{eqnarray}
and
\begin{eqnarray}
\label{eq40}
 \left\{
  \begin{array}{l}
    \nabla_H\cdot v_s=\eta-\frac{1}{|M|}\int_M\Phi dxdy,\\
    \nabla_H^{\bot}\cdot v_s=\theta,\quad\int_Mv_sdxdy=0,\quad v_s\text{ is periodic in }x,y,
  \end{array}
  \right.
\end{eqnarray}
where $\Phi$ is as in (\ref{eq16}).

We are going to estimate $ \int_{\Omega}|\nabla_Hv_T||\nabla_HT|^{q}dxdydz$ and $ \int_{\Omega}|\nabla_Hv_s||\nabla_HT|^{q}dxdydz$, separately.
For $v_s$, similar to (\ref{ADD24.5.5-10}), one has
\begin{equation}
\int_{\Omega}|\nabla_Hv_s||\nabla_HT|^{q}dxdydz
\leq\frac{1}{4}\big\||\nabla_HT|^{\frac{q-2}{2}}\partial_z\nabla_HT\big\|_2^2+C\|\nabla_H v_s\|_{L_z^1 L_M^\infty }^2\|\nabla_HT\|_q^q.\label{eq45}
\end{equation}
For $v_T$, it follows from the the elliptic estimate to (\ref{eq41}) and the definition of $\Phi$ (see (\ref{eq16})) that
\begin{eqnarray*}
&&\|\nabla_H^2 v_T \|_{L^q(M)}\leq C\|\nabla_H \Phi\|_{L^q(M)}\\
&=&C\left\|\int_{-h}^z \nabla_H Td\xi -\frac{1}{2h}\int_{-h}^h\int_{-h}^z\nabla_H Td\xi dz\right\|_{L^q(M)}
\leq C\|\nabla_H T\|_{q}.
\end{eqnarray*}
 Thanks to this, one deduces by Lemma \ref{lemma4} and Proposition \ref{pro3.1} that
 \begin{eqnarray*}
 &&\|\nabla_H v_T(.,z,t)\|_{L^\infty(M)} \\
 &\leq &C\left(\|\nabla_H\cdot v_T(.,z,t)\|_{L^\infty(M)}+\|\nabla_H^{\bot}\cdot v_T(.,z,t)\|_{L^\infty(M)}\right)\\
 &&\times\log\left(e+\|\nabla_H^2 v_T(.,z,t)\|_{L^q(M)}\right)\\
 &\leq &C\|T\|_{\infty}\log\left(e+\|\nabla_H T\|_q\right) \leq  C\log(e+\|\nabla_H T\|_q), \quad \forall z\in [-h,h],
\end{eqnarray*}
and thus
$\|\nabla_Hv_T\|_\infty\leq C\log(e+\|\nabla_H T\|_q).$
As a result, one has
 \begin{equation}
\int_{\Omega}|\nabla_Hv_T||\nabla_HT|^{q}dxdydz
\leq C\|\nabla_HT\|_q^q\log(e+\|\nabla_H T\|_q). \label{eq44}
\end{equation}

Substituting (\ref{eq45}) and (\ref{eq44}) into (\ref{eq42}) yields
\begin{eqnarray}
&&\frac{1}{q}\frac{d}{dt}\|\nabla_HT\|_q^q+\int_{\Omega}|\nabla_HT|^{q-2}|\partial_z\nabla_HT|^2dxdydz\nonumber\\
&\leq& C\|\nabla_H^2v\|_{L^1_z L^q_M }^2\|\nabla_H T\|_q^{q-2}+C\|\nabla_H v_s \|_{L_z^1 L_M^{\infty} }^2\|\nabla_HT\|_q^q\nonumber\\
&&+C\|\nabla_HT\|_q^q\log(e+\|\nabla_HT\|_q). \label{eq0415-2}
\end{eqnarray}

It remains to estimate $\int_{\mathcal T_0}^t\|\nabla_H^2v\|_{L^1_z L^q_M }^2ds$ and $\int_{\mathcal T_0}^t\|\nabla_H v_s \|_{ L_z^1 L_M^{\infty} }^2ds$.
By Proposition \ref{PRO-PARA-EST1}, it follows from the Minkowski inequality that
\begin{eqnarray}
 \int_{\mathcal T_0}^t\|\nabla_H^2v\|_{L^1_z L^q_M }^2ds&=&\|\nabla_H^2 v\|_{L^2_{(\mathcal T_0,t)}L_z^1L_M^q}^2\leq
 \|\nabla_H^2 v\|_{L_z^1 L^2_{(\mathcal T_0,t)}L_M^q}^2  \nonumber\\
 &\leq& C+C\int_{\mathcal T_0}^t\|\nabla_HT\|_q^2dt,\quad\forall t>\mathcal T_0.   \label{eq43}
\end{eqnarray}
By the Sobolev and Poincar\'e inequalities and applying the elliptic estimate to (\ref{eq40}) yield
\begin{align*}
\|\nabla_H v_s \|_{L^\infty(M)}&\leq C\|\nabla_H^2 v_s \|_{L^q(M)}
 \leq C\|\nabla_H(\eta, \theta)\|_{L^q(M)},
\end{align*}
from which, by the Minkowski and H\"older inequalities and Proposition \ref{pro4.3}, one deduces
\begin{eqnarray}
 \int_{\mathcal T_0}^\mathcal T \|\nabla_H v_s \|_{L_z^1 L_M^{\infty} }^2ds&=&\|\nabla_H v_s\|^2_{L^2_{(\mathcal T_0, \mathcal{T})} L_z^1 L^\infty_M}\leq C\|\nabla_H v_s\|^2_{L_z^2L^2_{(\mathcal T_0, \mathcal{T})}L^\infty_M} \nonumber\\
 &\leq&  C \|\nabla_H(\eta, \theta)\|^2_{L_z^2L^2_{(\mathcal T_0, \mathcal{T})}L^q_M}
\leq C.\label{eq0415-1}
\end{eqnarray}

 Denote $A_2$, $B_2$, $\varpi$, and $E_2$ as
 \begin{align*}
 &A_2:=(\|\nabla_H T\|_q^q+e)^{\frac{1}{q}},\quad B_2:=\displaystyle\int_{\Omega}|\nabla_HT|^{q-2}|\partial_z\nabla_HT|^2dxdydz,\\
 &\varpi:=\|\nabla_H v_s \|_{L_z^1 L_M^\infty }^2, \quad  E_2:=\|\nabla_H^2v\|_{L^1_z L^q_M }^2.
 \end{align*}
Dividing both sides of (\ref{eq0415-2}) by $A_2^{q-2}$, and recalling (\ref{eq43}) and (\ref{eq0415-1}), one obtains
 \begin{equation*}
  \left\{
  \begin{array}{l}
 \frac{d}{dt}A_2^2+\frac{B_2}{A_2^{q-2}}\leq CE_2+CA_2^2\log A_2+C\varpi A_2^2,\\
  \int_{\mathcal{T}_0}^t E_2(s)ds\leq C+\int_{\mathcal{T}_0}^t A_2^2(s)ds,\quad\mbox{with }
   \int_{\mathcal T_0}^\mathcal T\varpi dt\leq C.
  \end{array}
  \right.
\end{equation*}
Therefore, by Lemma \ref{lemma5}, the conclusion follows.
\end{proof}

\subsection{Overall a priori estimates}

\begin{proposition}
  \label{cor1}
Suppose the assumptions of Proposition \ref{prop2.1} hold and $\mathcal T$ is a given positive time. Then, the following estimate holds
\begin{eqnarray*}
\sup_{0\leq t \leq \mathcal T} \|(\nabla_HT, \partial_z T)\|_2^2(t)
+ \int_0^{\mathcal T}\|(\partial_z^2 T, \partial_z\nabla_H T)\|_2^2dxdydzdt\leq C,
\end{eqnarray*}
where C is a positive constant depending only $q$, $h$, $M_0$, and $\mathcal T$.
\end{proposition}

\begin{proof}
Following the arguments in Propositions \ref{pro3.6}--\ref{pro4.4}, actually by much simpler calculations, one obtains
\begin{align}
\sup_{0\leq t \leq \mathcal T} \|\nabla_HT\|_2^2 + \int_0^{\mathcal T}\|\partial_z\nabla_H T\|_2^2dxdydzdt\leq C. \label{eq327-0}
\end{align}
Multiplying (\ref{eq10}) by $-\partial_z^2 T$, it follows from integration by parts that
\begin{eqnarray}
&&\frac{1}{2}\frac{d}{dt}\|\partial_zT\|_2^2+\|\partial_z^2T\|_2^2=\int_{\Omega}(v\cdot\nabla_HT+w\partial_zT)\partial_z^2Tdxdydz\nonumber\\
&=&-\int_{\Omega}\left(v\cdot\partial_z\nabla_H T\partial_zT+u\cdot\nabla_HT\partial_zT-\frac{1}{2}(\nabla_H\cdot v)|\partial_zT|^2\right)dxdydz.\label{eq37}
\end{eqnarray}
Noticing that $ |v|\leq \int_{-h}^h\left(\frac{|v|}{2h} +|\partial_zv|\right)dz$ and $ |\nabla_H T|\leq\int_{-h}^h|\partial_z\nabla_H T|dz$, it follows from Lemma \ref{lemma1} and the Young inequality that
\begin{eqnarray*}
&&\left|\int_{\Omega}\left(v\cdot\partial_z\nabla_H T\partial_zT+u\cdot\nabla_HT\partial_zT\right)dxdydz\right|\nonumber\\
&\leq&\frac{1}{2h} \int_M\left(\int_{-h}^h|v|dz\right)\left(\int_{-h}^h|\partial_z\nabla_H T||\partial_zT|dz\right)dxdy\nonumber\\
&&+\int_M\left(\int_{-h}^h|u|dz\right)\left(\int_{-h}^h|\partial_z\nabla_H T||\partial_zT|dz\right)dxdy\nonumber\\
&&+\int_M\left(\int_{-h}^h|\partial_z\nabla_H T|dz\right)\left(\int_{-h}^h|u||\partial_zT|dz\right)dxdy\nonumber\\
&\leq& C\|\partial_z\nabla_HT\|_2\|v\|_2^{\frac{1}{2}}\|(v,\nabla_Hv)\|_2^{\frac{1}{2}})\|\partial_zT\|_2^{\frac{1}{2}}
\|(\partial_zT,\partial_z\nabla_HT)\|_2^{\frac{1}{2}}\nonumber\\
&&+C\|\partial_z\nabla_HT\|_2\|u\|_2^{\frac{1}{2}}\|(u,\nabla_Hu)\|_2^{\frac{1}{2}}\|\partial_zT\|_2^{\frac{1}{2}}
\|(\partial_zT,\partial_z\nabla_HT)\|_2^{\frac{1}{2}}\nonumber\\
&\leq& \frac{1}{4}\|\partial_z\nabla_HT\|_2^2+C(\|v\|_2^2\|\nabla_Hv\|_2^2+\|v\|_2^4\nonumber\\
&&+\|u\|_2^2\|\nabla_Hu\|_2^2+\|u\|_2^4+1)\|\partial_zT\|_2^2.\label{eqTz1}
\end{eqnarray*}
Consequently, by Propositions \ref{pro3.1}--\ref{pro4.2} and Corollary \ref{COR1-1}, one has
\begin{eqnarray}
\left|\int_{\Omega} \left(v\cdot\partial_z\nabla_H T\partial_zT+u\cdot\nabla_HT\partial_zT\right)dxdydz\right|\nonumber\\
 \leq \frac{1}{4}\|\partial_z\nabla_HT\|_2^2+C(1+\|\nabla_H u\|_2^2)\|\partial_zT\|_2^2.\label{eq327-1}
\end{eqnarray}
Noticing that $T|_{z=-h}=T|_{z=h}=0$, integration by parts, and using the H\"{o}lder inequality yield
\begin{align*}
\int_{\Omega}|\partial_zT|^4dxdydz&=-\int_{\Omega}\partial_z(|\partial_zT|^2\partial_zT)Tdxdydz\\
&\leq3\|T\|_{\infty}\|\partial_z^2T\|_2\|\partial_zT\|_4^2,
\end{align*}
from which, using Proposition \ref{pro3.1}, one gets
\begin{align}
\|\partial_zT\|_4^2\leq 3\|T\|_{\infty}\|\partial_z^2T\|_2\leq C\|\partial_z^2T\|_2. \label{eq327-2}
\end{align}
Recalling the definition of $\eta$ (see (\ref{eq16})), one deduces by the H\"{o}lder and Young inequalities and (\ref{eq327-2}) that
\begin{eqnarray*}
\int_{\Omega}(\nabla_H\cdot v)|\partial_zT|^2dxdydz= \int_{\Omega}(\eta-\Phi)|\partial_zT|^2dxdydz\nonumber\\
\leq \|\eta\|_2\|\partial_zT\|_4^2+\|\Phi\|_{\infty}\|\partial_zT\|_2^2.\label{eqTz2}
\end{eqnarray*}
As a result, by Propositions \ref{pro3.1}--\ref{pro4.2}, one obtains
\begin{align}
\int_{\Omega}(\nabla_H\cdot v)|\partial_zT|^2dxdydz
\leq& \frac{1}{4}\|\partial_z^2T\|_2^2+C(1+\|\partial_zT\|_2^2).\label{eq327-3}
\end{align}
Substituting (\ref{eq327-1}) and (\ref{eq327-3}) into (\ref{eq37}) yields
\begin{align*}
\frac{d}{dt}\|\partial_zT\|_2^2+\|\partial_z^2T\|_2^2\leq C(1+\|\nabla_Hu\|_2^2)(\|\partial_zT\|_2^2+1)
\end{align*}
which, by the Gr\"{o}nwall inequality and Proposition \ref{pro4.2}, implies
\begin{eqnarray*}
\sup_{0\leq t \leq\mathcal T} \|\partial_z T\|_2^2 + \int_0^{\mathcal T}\|\partial_z^2 T\|_2^2dt\leq C.
\end{eqnarray*}
Combining this with (\ref{eq327-0}), the conclusion follows.
\end{proof}

\begin{proposition}
  \label{pro3.8}
Suppose the assumptions of Proposition \ref{prop2.1} hold and $\mathcal T$ is a given positive time. Then, it holds that
\begin{align*}
\int_0^{\mathcal T}(\|\nabla_H^2v\|_2^2+\|\partial_t v\|_2^2+\| \partial_t T\|_2^2) \leq C
\end{align*}
for a positive constant $C$ depending only $q$, $h$, $M_0$, and $\mathcal T$.
\end{proposition}
\begin{proof}

By (\ref{eq8}) and applying the elliptic estimate to (\ref{eq15}), one has
\begin{align*}
&\|\partial_t v\|_2\leq \|v\cdot\nabla_Hv\|_2+\|w u\|_2+\|\Delta_H v\|_2+\|f_0\vec{k}\times v\|_2+\left\|\nabla_H\left(\int_{-h}^zTd\xi\right)\right\|_2\nonumber\\
&\quad+\|\nabla_H p_s\|_2
\leq C\left(\big\||v|\nabla_Hv\big\|_2+\|w u\|_2+\|\Delta_H v\|_2+\|\nabla_H T\|_2+\|v\|_2\right)
\end{align*}
and thus, by Propositions \ref{pro3.1} and \ref{cor1}, it follows that
\begin{align}
\|\partial_t v\|_2\leq C\left(1+\big\||v|\nabla_Hv\big\|_2+\|w u\|_2+\|\Delta_H v\|_2\right).\label{eq38}
\end{align}
By the H\"{o}lder, Gagliardo-Nirenberg, and Poincar\'e inequalities, Proposition \ref{pro3.1}, and Corollary \ref{COR1-1}, one deduces that
\begin{eqnarray}
\big\||v|\nabla_Hv\big\|_2^2&\leq& C\|v\|_6^2\|\nabla_Hv\|_3^2\leq C\|\nabla_Hv\|_2\|\nabla\nabla_Hv\|_2 \nonumber\\
&\leq& C(\|\nabla_H^2v\|_2+\|\nabla_H u\|_2).\label{eqloc1}
\end{eqnarray}
It follows from the H\"{o}lder, Minkowski, and Ladyzhenskaya inequalities that
\begin{eqnarray*}
\|w u\|_2^2
&=&\int_{\Omega}|w u|^2dxdydz=\int_{\Omega}\left(\int_{-h}^z(\eta-\Phi)d\xi\right)^2|u|^2dxdydz\nonumber\\
&\leq&2\int_M\left(\int_{-h}^h|\eta| dz\right)^2\left(\int_{-h}^h|u|^2dz\right)dxdy\nonumber\\
&&+2\int_M\left(\int_{-h}^h|\Phi| dz\right)^2\left(\int_{-h}^h|u|^2dz\right)dxdy\nonumber\\
&\leq&C\left(\int_{-h}^h\|\eta \|_{4,M}dz\right)^2\int_{-h}^h\|u \|_{4,M}^2dz+C\|T\|_{\infty}^2\|u\|_2^2\nonumber\\[0.4em]
&\leq& C\left(\|\eta\|_2\|(\eta,\nabla_H\eta)\|_2\|u\|_2\|(u,\nabla_Hu)\|_2+\|u\|_2^2\right),\label{eqvt3}
\end{eqnarray*}
where $\theta, \eta,$ and $\Phi$ are as in (\ref{eq16}). From the above, one deduces by Proposition \ref{pro4.2} that
\begin{align}
\|w u\|_2^2\leq C(1+\|\nabla_H\eta\|_2)(1+\|\nabla_Hu\|_2).\label{eqloc2}
\end{align}
Substituting (\ref{eqloc1}) and (\ref{eqloc2}) into (\ref{eq38}), it follows from Propositions \ref{pro4.2} and
\ref{cor1} that
\begin{eqnarray*}
\int_0^{\mathcal T} \|(\partial_t v,\nabla_H^2v)\|_2^2dt&\leq& C\int_0^{\mathcal T}\left(1+\|(\nabla_H \eta,\nabla_H u,\nabla_H^2v)\|_2^2\right)dt\\
&\leq& C+C\int_0^{\mathcal T}\|(\nabla_H \eta,\nabla_H u,\nabla_H\theta,\nabla_H T)\|_2^2 dt\leq C,
\end{eqnarray*}
where the following estimate guaranteed by elliptic estimate was used
\begin{align}
\|\nabla_H^2v\|_2^2\leq & C\left(\|\nabla_H\nabla_H\cdot v\|_2^2+\|\nabla_H \nabla_H^{\bot}\cdot v\|_2^2\right)\nonumber\\
\leq& C\left(\|\nabla_H \eta\|_2^2+\|\nabla_H\theta\|_2^2+\|\nabla_H T\|_2^2\right). \label{eq327-4}
\end{align}

By (\ref{eq10}), we have
\begin{eqnarray}
&&\|\partial_t T\|_2\leq \|v\cdot\nabla_HT\|_2+\|w \partial_zT\|_2+\|\partial_z^2 T\|_2. \label{eq39}
\end{eqnarray}
Noticing that $|\nabla_H T|^2\leq 2\int_{-h}^h |\nabla_H T||\nabla_H\partial_z T|dz$, it follows from the H\"{o}lder and Gagliardo-Nirenberg inequalities and (\ref{eq327-4}) that
 \begin{eqnarray*}
\|v\cdot\nabla_HT\|_2^2&=&\int_M\int_{-h}^h|v|^2|\nabla_HT|^2dzdxdy\nonumber\\
&\leq& 2\int_M\left(\int_{-h}^h|v|^2dz\right)\left(\int_{-h}^h|\nabla_HT||\partial_z\nabla_HT|dz\right)dxdy\nonumber\\
&\leq& C\left\|\int_{-h}^h|v|^2dz \right\|_{L^\infty(M)}\|\nabla_HT\|_2\|\nabla_H\partial_zT\|_2\nonumber\\
&\leq& C\int_{-h}^h\|v \|_{\infty,M}^2dz\|\nabla_HT\|_2\|\nabla_H\partial_zT\|_2\nonumber\\[0.3em]
&\leq& C\|v\|_2(\|v\|_2+\|\nabla_H^2v\|_2)\|\nabla_HT\|_2\|\nabla_H\partial_zT\|_2\nonumber\\[0.5em]
&\leq& C\|v\|_2\|(v,\nabla_H \theta,\nabla_H \eta,\nabla_H T)\|_2\|\nabla_HT\|_2\|\nabla_H\partial_zT\|_2.\label{eqTt1}
\end{eqnarray*}
Thus, by Propositions \ref{pro3.1} and \ref{cor1}, one reaches
\begin{align}
\|v\cdot \nabla_H T\|_2^2\leq C\left(1+\|\nabla_H \theta\|_2+\|\nabla_H\eta\|_2\right)\|\nabla_H\partial_z T\|_2. \label{eq327-5}
\end{align}
Same arguments as those for $\|w u\|_2^2$ yield
 \begin{align*}
\|w \partial_zT\|_2^2 \leq  C\left[\|\eta\|_2\|(\eta,\nabla_H\eta)\|_2\|\partial_zT\|_2\|(\partial_zT,\nabla_H\partial_zT)\|_2+\|\partial_zT\|_2^2\right].
\end{align*}
Thus, by Propositions \ref{pro4.2} and \ref{cor1}, one has
\begin{align}
\|w \partial_zT\|_2^2\leq C(1+\|\nabla_H \eta\|_2)(1+\|\nabla_H\partial_z T\|_2). \label{eq327-6}
\end{align}
Substituting (\ref{eq327-5}) and (\ref{eq327-6}) into (\ref{eq39}) and by Propositions \ref{pro4.2} and \ref{cor1}, one deduces
$$\int_0^{\mathcal T} \|\partial_t T\|_2^2dt\leq C\int_0^{t^{**}}(1+\|\nabla_H \eta\|_2^2+\|\nabla_H \theta\|_2^2+\|\nabla_H\partial_z T\|_2^2)dt\leq C.$$
This completes the proof.
\end{proof}

Consequently, it is concluded by Propositions \ref{pro3.1}--\ref{pro4.2} and \ref{pro4.4}--\ref{pro3.8} that

\begin{corollary}
  \label{cor5.1}
Let $(v,T)$ be the unique global strong solution in Proposition \ref{prop2.1} and let $M_0$ be a constant such that (\ref{M0}) holds.
Then, for any given positive time $\mathcal{T}$, the following a priori estimate holds
\begin{align*}
  \sup_{0\leq t\leq \mathcal{T}}(\|(v, T)\|_{H^1}^2+\|(u, \nabla_H T)\|_q^q)
  +\int_{0}^{\mathcal{T}}\left(\|(\nabla_H v, \partial_z T)\|_{H^1}^2+\|(\partial_t v, \partial_t T)\|_2^2\right)dt \leq C
\end{align*}
for a positive constant $C$ depending only on $q$, $h$, $M_0$, and $\mathcal{T}$.
\end{corollary}

\section{Proof of Theorem \ref{the1}}\label{sec4}

\begin{proof}[\bf Proof of Theorem \ref{the1}. Global existence.\rm]

Denote $M=\|(v_0, T_0)\|_{X_1}$ where $X_1$ is the norm given in (\ref{normX1}).
Choose an approximating sequence periodic functions $v_0^N$ and $T_0^N$, for $N=1, 2, \cdots,$ such that
$v_0^N \in H^2(\Omega)$, $T_0^N \in H^1(\Omega)\cap L^{\infty}(\Omega)$, with $\int_{-h}^h \nabla_H \cdot v_0^Ndz=0$ and $\nabla_H T_0^N \in L^q(\Omega)$, for some $q \in (2,\infty)$,
and that $v_0^N$ and $T_0^N$ are even and odd in $z$, respectively, satisfying
\begin{eqnarray*}
  & (v_0^N, T_0^N)\rightarrow (v_0, T_0) \, &\text{as} \, N\rightarrow \infty, \quad in\, H^1(\Omega),
\end{eqnarray*}
and $\|(v_0^N, T_0^N)\|_{X_1} \leq 2M$.
Then, by Proposition \ref{prop2.1}, there exists a unique global solution $(v^N, T^N)$ to system (\ref{eq8})--(\ref{eq10}), subject to (\ref{eq12})--(\ref{eq14}), with initial data $(v_0^N, T_0^N)$.\par
Since $\|(v_0^N, T_0^N)\|_{X_1} \leq 2M$, Corollary \ref{cor5.1} implies that
$$
  \sup_{0\leq t\leq \mathcal{T}}(\|(v^N, T^N)\|_{H^1}^2(t)+\|(u^N, \nabla_H T^N)\|_q^q(t)) \leq  C
  $$ and
  $$
  \int_{0}^{\mathcal{T}}\left(\|(\nabla_H v^N, \partial_z T^N)\|_{H^1}^2+\|(\partial_t v^N, \partial_t T^N)\|_2^2\right)dt  \leq C,
  $$
  where $C$ is a positive constant depending only on $q, h, M,$ and $\mathcal T$, but is independent of $N$.

By the Aubin-Lions lemma (Lemma \ref{lemma7}), there is a subsequence, still denote by $(v^N, T^N)$, and $(v, T)$, such that, as $N\rightarrow \infty$,
\begin{eqnarray*}
  &(v^N, T^N)\rightarrow (v,T)  \,  &\text{in} \, C([0, \mathcal{T}]; L^2(\Omega)),\\
  &(v^N, T^N)\overset{*}{\rightharpoonup}(v, T)\, &\text{in}\, L^{\infty}(0,\mathcal{T}; H^1(\Omega)),\\
  &(u^N, \nabla_H T^N)\overset{*}{\rightharpoonup}(u, \nabla_HT)\,&\text{in}\, L^{\infty}(0,\mathcal{T}; L^q(\Omega)),\\
  &(\nabla_Hv^N, \partial_zT^N)\rightharpoonup(\nabla_Hv, \partial_zT)\,&\text{in}\, L^2(0,\mathcal{T}; H^1(\Omega)),\\
  &(\partial_t v^N, \partial_t T^N)\rightharpoonup (\partial_t v, \partial_t T)\,&\text{in}\,L^2(0,\mathcal{T};L^2(\Omega)),
\end{eqnarray*}
where $\rightharpoonup$ and $\overset{*}{\rightharpoonup}$ denote the weak and weak-* convergences, respectively.
Thanks to these convergences, one can show easily that $(v,T)$ is a strong solution to system (\ref{eq8})--(\ref{eq10}) subject to (\ref{eq12})--(\ref{eq14}) on $\Omega\times (0,\mathcal{T})$.

 \bf Continuous dependence on the initial data. \rm
 Let $(v_1, T_1)$ and $(v_2, T_2)$ be two strong solutions system (\ref{eq8})--(\ref{eq10}), subject to (\ref{eq12})--(\ref{eq14}), with initial data $(v_{01}, T_{01})$ and $(v_{02}, T_{02})$, respectively. Let $w_i$ be given by (\ref{eq11}) corresponding to $v_i$, $i=1,2$. Denote $v=v_1-v_2, w=w_1-w_2,$ and $T=T_1-T_2$. Then, $(v,T)$ satisfies
 \begin{eqnarray}
&&\label{eqcontv}
\begin{array}{l}
  \displaystyle\partial_tv+(v_1\cdot\nabla_H)v+w_1\partial_zv-\Delta_Hv+f_0\vec{k}\times v +\nabla_Hp_s(x,y,t)\\
  \,\,\,\,\,\,\displaystyle=\int_{-h}^z\nabla_HT(x,y,\xi,t)\mathrm{d}\xi-(v\cdot\nabla_H)v_2-w\partial_zv_2,
\end{array}\\
&& \qquad\int_{-h}^h \nabla_H\cdot vdz=0 \nonumber\\
&& \partial_tT+v_1\cdot\nabla_HT+w_1\partial_zT-\partial_z^2T=-v\cdot \nabla_H T_2-w\partial_z T_2, \label{eqcontT}
\end{eqnarray}
subject to the boundary conditions (\ref{eq12})--(\ref{eq13}), with initial data
$$(v, T)|_{t=0}=(v_{01}-v_{02}, T_{01}-T_{02}).$$

Note that
$
|\nabla_H v_2|\leq \frac{1}{2h}\int_{-h}^h|\nabla_H v_2|\mathrm{d}\xi+\int_{-h}^h|\nabla_H\partial_zv_2|\mathrm{d}\xi.
$
Multiplying (\ref{eqcontv}) with $v$ and integrating over $\Omega$, it follows from integration by parts, the H\"{o}lder inequality and Lemma \ref{lemma1} that
\begin{align*}
&\frac{1}{2}\frac{d}{dt}\|v\|_2^2+\|\nabla_Hv\|_2^2\\
=&\int_{\Omega}\left[\nabla_H\left(\int_{-h}^zTd\xi\right)-(v\cdot\nabla_H)v_2-w\partial_zv_2\right]\cdot vdxdydz\\
=&-\int_{\Omega}\left[\left(\int_{-h}^zTd\xi\right)(\nabla_H\cdot v)+((v\cdot\nabla_H )v_2+w\partial_zv_2)\cdot v\right]dxdydz\\
\leq& C\|T\|_2\|\nabla_H v\|_2+C\int_{M}\left(\int_{-h}^h|v|^2dz\right)\left(\int_{-h}^h(|\nabla_H\partial_zv_2|+|\nabla_Hv_2|)dz\right)dxdy\\
&+C\int_M \left(\int_{-h}^h|\nabla_H\cdot v|dz\right)\left(\int_{-h}^h|\partial_zv_2||v|dz\right)dxdy\\
\leq& C\|T\|_2\|\nabla_H v\|_2+C\|v\|_2(\|v\|_2+\|\nabla_H v\|_2)(\|\nabla_H v_2\|_2+\|\nabla_H\partial_zv_2\|_2)\\
&+C\|\nabla_H v\|_2\|\partial_z v_2\|_2^{\frac{1}{2}}(\|\partial_z v_2\|_2^{\frac{1}{2}}+\|\nabla_H\partial_z v_2\|_2^{\frac{1}{2}})\|v\|_2^{\frac{1}{2}}(\|v\|_2^{\frac{1}{2}}+\|\nabla_Hv\|_2^{\frac{1}{2}})\\
\leq&\frac{1}{2}\|\nabla_H v\|_2^2+C(1+\|\nabla v_2\|_2^2)^2(1+\|\nabla_H\partial_z v_2\|_2^2)(\|T\|_2^2+\|v\|_2^2),
\end{align*}
that is,
\begin{align}
\frac{d}{dt}\|v\|_2^2+\|\nabla_Hv\|_2^2\leq C(1+\|\nabla v_2\|_2^2)^2(1+\|\nabla_H\partial_z v_2\|_2^2)(\|T\|_2^2+\|v\|_2^2).\label{eqcondv}
\end{align}

Multiplying (\ref{eqcontT}) by $T$, it follows from integration by part that
\begin{align}
\frac{1}{2}\frac{\mathrm{d}}{\mathrm{d}t}\|T\|_2^2+\|\partial_z T\|_2^2=-\int_{\Omega}(v\cdot \nabla_HT_2+w\partial_z T_2)T\mathrm{d}x\mathrm{d}y\mathrm{d}z. \label{eqconpT}
\end{align}
Noticing $|T|\leq\int_{-h}^h|\partial_zT|dz$, guaranteed by $T|_{z=-h}=0$, it follows from the Minkowski, H\"{o}lder, Gagliardo-Nirenberg, and Young inequalities that
 \begin{eqnarray}
&&-\int_{\Omega}v\cdot \nabla_HT_2Tdxdydz\nonumber\\
&\leq&\int_M\left(\int_{-h}^h|v||\nabla_HT_2|dz\right)\left(\int_{-h}^h|\partial_zT|dz\right)dxdy\nonumber\\
&\leq&\int_M\left(\int_{-h}^h|v|^2dz\right)^{\frac{1}{2}}\left(\int_{-h}^h|\nabla_HT_2|^2dz\right)^{\frac{1}{2}}\left(\int_{-h}^h|\partial_z T|dz\right)dxdy\nonumber\\
&\leq & \left\|\left(\int_{-h}^h|v|^2dz\right)^{\frac{1}{2}}\right\|_{L^{\frac{2q}{q-2}}(M)}\left\|\left(\int_{-h}^h|\nabla_HT_2|^2dz\right)^{\frac{1}{2}}\right\|_{L^q(M)}
\left\|\left(\int_{-h}^h|\partial_z T|dz\right)\right\|_{L^2(M)}\nonumber\\
&\leq& \left(\int_{-h}^h\|v\|_{\frac{2q}{q-2},M}^2dz\right)^{\frac{1}{2}}\left(\int_{-h}^h \| \nabla_H T_2 \|_{q,M}^2dz\right)^{\frac{1}{2}}\left(\int_{-h}^h\|\partial_z T\|_{2,M}dz\right)\nonumber\\
&\leq& C\left[\int_{-h}^h\left(\|v\|_{2,M}^2+\|v\|_{2,M}^{\frac{2(q-2)}{q}}\|\nabla_H v\|_{2,M}^{\frac{4}{q}}\right)dz\right]^{\frac{1}{2}}\|\nabla_H T_2\|_q\|\partial_zT\|_2\nonumber\\
&\leq& C\left(\|v\|_2+\|v\|_2^{\frac{q-2}{q}}\|\nabla_H v\|_2^{\frac{2}{q}}\right)\|\nabla_H T_2\|_q\|\partial_zT\|_2\nonumber\\
&\leq&\frac{1}{4}(\|\partial_zT\|_2^2+\|\nabla_H v\|_2^2)+C\left(\|\nabla_H T_2\|_q^2+\|\nabla_H T_2\|_q^{\frac{2q}{q-2}}\right)\|v\|_2^2. \label{eqconPt1}
\end{eqnarray}
For the term $-\int_{\Omega}w\partial_z T_2T\mathrm{d}x\mathrm{d}y\mathrm{d}z$, integration by part yields
\begin{eqnarray}
-\int_{\Omega}w\partial_z T_2Tdxdydz&=&\int_{\Omega}T_2(\partial_zw T+w\partial_z T)dxdydz\nonumber\\
&\leq& \|T_2\|_{\infty}(\|\nabla_H v\|_2\|T\|_2+\|\partial_z T\|_2\|\nabla_H v\|_2)\nonumber\\
&\leq& \frac{1}{4}\|\partial_z T\|_2^2+C(\|\nabla_H v\|_2^2+\|T\|_2^2).\label{eqconPt2}
\end{eqnarray}

Substituting (\ref{eqconPt1}) and (\ref{eqconPt2}) into (\ref{eqconpT}) leads to
\begin{align}
&\frac{\mathrm{d}}{\mathrm{d}t}\|T\|_2^2+\|\partial_z T\|_2^2 \nonumber\\
\leq& C\|\nabla_Hv\|_2^2+C\left(1+\|\nabla_H T_2\|_q^2+\|\nabla_H T_2\|_q^{\frac{2q}{q-2}}\right)(\|v\|_2^2+\|T\|_2^2). \label{eqcontTz}
\end{align}
Multiplying (\ref{eqcondv}) by a sufficiently large positive constant $K$, and summing the resultant with (\ref{eqcontTz}) yield
\begin{eqnarray*}
&&\frac{\mathrm{d}}{\mathrm{d}t}(K\|v\|_2^2+\|T\|_2^2)+\frac{1}{2}(K\|\nabla_Hv\|_2^2+\|\partial_z T\|_2^2)\\
&&\leq C\left[(1+\|\nabla v_2\|_2^2)^2(1+\|\nabla_H\partial_z v_2\|_2^2)+\|\nabla_H T_2\|_q^2+\|\nabla_H T_2\|_q^{\frac{2q}{q-2}}\right](\|T\|_2^2+\|v\|_2^2),
\end{eqnarray*}
from which, the conclusion follows by the Gr\"{o}nwalll inequality.
\end{proof}
\section*{Acknowledgments}
The work of J.L. was supported in part by the the National
Natural Science Foundation of China (Grant No. 12371204) and the Key Project of National Natural Science Foundation of China (Grant No. 12131010). The work of E.S.T. was supported in part by the DFG Research Unit FOR 5528 on Geophysical Flows. Moreover, the research of E.S.T. has also benefited from the inspiring environment of the CRC 1114 ``Scaling Cascades in Complex Systems", Project Number 235221301, Project C06, funded by Deutsche Forschungsgemeinschaft (DFG).

\end{document}